\documentclass[english]{svjour3}
\usepackage[T1]{fontenc}
\usepackage[latin9]{inputenc}
\usepackage[a4paper]{geometry}
\usepackage{babel}
\usepackage{url}
\usepackage{amsmath}
\usepackage{amssymb}
\usepackage{graphicx}
\usepackage{esint}
\usepackage[unicode=true,
 bookmarks=true,bookmarksnumbered=true,bookmarksopen=false,
 breaklinks=true,pdfborder={0 0 1},backref=false,colorlinks=false]
 {hyperref}
\hypersetup{pdftitle={Maxwell Time and Conjgate Loci},
 pdfauthor={Yasir Awais Butt}}
\usepackage{breakurl}

\makeatletter

\numberwithin{equation}{section}

\usepackage{authblk}
\smartqed  
\numberwithin{equation}{section}
\numberwithin{theorem}{section}
\numberwithin{proposition}{section}

\makeatother

\begin{document}

\title{Cut Locus and Optimal Synthesis in Sub-Riemannian Problem on the
Lie Group SH(2)}

\author{Yasir Awais Butt, Yuri L. Sachkov, Aamer Iqbal Bhatti}

\institute{Yasir Awais Butt \at Department of Electronic Engineering\\
Muhammad Ali Jinnah University\\
Islamabad, Pakistan\\
Tel.: +92-51-111878787\\
\email{yasir\_awais2000@yahoo.com}\\
\and Yuri L. Sachkov \at Program Systems Institute\\
Pereslavl-Zalessky, Russia\\
\email{yusachkov@gmail.com}\\
\and Aamer Iqbal Bhatti\at Department of Electronic Engineering\\
Muhammad Ali Jinnah University\\
Islamabad, Pakistan\\
Tel.: +92-51-111878787\\
\email{aib@jinnah.edu.pk}}
\maketitle
\begin{abstract}
Global optimality analysis in sub-Riemannian problem on the Lie group SH(2) is considered. We cutout open dense domains in the preimage and in the image of the exponential mapping based on the description of Maxwell strata. We then prove that the exponential mapping restricted to these domains is a diffeomorphism. Based on the proof of diffeomorphism,
the cut time, i.e., time of loss of global optimality is computed on $\mathrm{\ensuremath{SH(2)}}$. We also consider the global structure of the exponential mapping and obtain an explicit description of cut locus and optimal synthesis. 


\end{abstract}

\section{Introduction}
In this work we complete our study of the sub-Riemannian problem on the Lie group $\mathrm{SH}(2)$ which is the group of motions of pseudo Euclidean plane. The work was initiated in \cite{Extremal_Pseudo_Euclid} where we defined the sub-Riemannian problem. The control system comprises two 3-dimensional left invariant vector fields and a 2-dimensional
linear control vector. We applied PMP to the control system and obtained the corresponding Hamiltonian system. In \cite{intg_SH2} we proved the Liouville integrability of the Hamiltonian system. We calculated the Hamiltonian flow such that the extremal trajectories were parametrized
in terms of Jacobi elliptic functions \cite{Extremal_Pseudo_Euclid}. Since PMP states only the first order optimality conditions, the trajectory resulting from PMP are only potentially optimal called extremal trajectories
or geodesics. Further analysis based on second order optimality conditions is then needed to segregate the optimal trajectories or the minimizing geodesics. It is well known that the candidate optimal trajectories lose optimality either at the Maxwell points or at the conjugate points
\cite{agrachev_sachkov},\cite{max_sre},\cite{cut_sre1}. Based on the optimality analysis one is able to state the time of loss of global optimality known as the cut time. Rigorous techniques for this optimality analysis have evolved over the years from research on related sub-Riemannian problems on various Lie groups, see e.g., \cite{max_sre}, \cite{cut_sre1},
\cite{Sachkov_Dido_Comp_Max}, \cite{cut_engel}. These techniques were employed in \cite{Extremal_Pseudo_Euclid} and \cite{Max_Conj_SH2} to compute the Maxwell strata and the conjugate locus in the problem under investigation. An effective upper bound on the cut time was also computed. 

In this paper we extend the global optimality analysis similar to \cite{cut_sre2}. We decompose the image $M=\mathrm{SH}(2)$ and the preimage of the exponential mapping into open dense sets based on the Maxwell strata and conjugate loci and prove that the exponential mapping between these sets is a diffeomorphism. This leads naturally to the proof that the cut time is equal to the first Maxwell time. Finally, we analyze the global structure of the exponential mapping and obtain explicit characterization of the cut locus and the optimal synthesis on the manifold $\mathrm{SH}(2)$. 

The paper is organized as follows. In Section 2, we review the results from \cite{Extremal_Pseudo_Euclid} and \cite{Max_Conj_SH2} as ready reference. Sections 3 and 4 contain the main results of this work. In Section 3 we state and prove the conditions for exponential mapping being a diffeomorphism and compute the cut time. Section 4 pertains to explicit characterization of the Maxwell strata and the cut locus
in terms of a stratification of $\mathrm{SH}(2)$. In Section 5 we conclude this work. 

\section{Previous Work}

\subsection{Problem Statement\label{sec:Problem-Statement}}

The Lie group $\mathrm{SH(2)}$ is a 3-dimensional group of roto-translations
of the pseudo Euclidean plane \cite{Ja.Vilenkin}. The sub-Riemannian
problem on the Lie group $\mathrm{SH(2)}$ reads as follows \cite{Extremal_Pseudo_Euclid}:
\begin{eqnarray}
\dot{x} & = & u_{1}\cosh z,\quad\dot{y}=u_{1}\sinh z,\quad\dot{z}=u_{2},\label{eq:2.1}\\
q & = & (x,y,z)\in M=\mathrm{SH(2)\cong\mathbb{R}^{3}},\quad x,y,z\in\mathbb{R},\quad(u_{1},u_{2})\in\mathbb{R}^{2},\label{eq:2.2}\\
q(0) & = & (0,0,0),\qquad q(t_{1})=q_{1}=(x_{1},y_{1},z_{1}),\label{eq:2.3}\\
l & = & \int_{0}^{t_{1}}\sqrt{u_{1}^{2}+u_{2}^{2}}\, dt\to\min.\label{eq:2.4}
\end{eqnarray}
By Cauchy-Schwarz inequality, the sub-Riemannian length functional
$l$ minimization problem (\ref{eq:2.4}) is equivalent to the problem
of minimizing the following action functional with fixed $t_{1}$
\cite{sachkov_lectures}:
\begin{equation}
J=\frac{1}{2}\intop_{0}^{t_{1}}(u_{1}^{2}+u_{2}^{2})dt\rightarrow\min.\label{eq:2.5}
\end{equation}

\subsection{Known Results\label{sec:Previous-Work}}

We now briefly review the results from \cite{Extremal_Pseudo_Euclid} and \cite{Max_Conj_SH2} as a ready reference in this paper. System (\ref{eq:2.1}) satisfies the bracket generating condition and is hence globally controllable \cite{Chow},\cite{Ravchevsky}. Existence of optimal trajectories for the optimal control problem (\ref{eq:2.1})--(\ref{eq:2.5})
follows from Filippov\textquoteright s theorem \cite{agrachev_sachkov}.
We applied PMP \cite{agrachev_sachkov} to (\ref{eq:2.1})--(\ref{eq:2.5}) to derive the normal Hamiltonian system. It turns out that the vertical part of the normal Hamiltonian system is a double covering of a mathematical pendulum. The normal Hamiltonian system is given as:

\begin{eqnarray}
\dot{\gamma} & = & c,\quad\dot{c}=-\sin\gamma,\quad\lambda=(\gamma,c)\in C\cong(2S_{\gamma}^{1})\times\mathbb{R}_{c},\quad2S_{\gamma}^{1}=\mathbb{R}/(4\pi\mathbb{Z}),\label{eq:2.6}\\
\dot{x} & = & \cos\frac{\gamma}{2}\cosh z,\quad\dot{y}=\cos\frac{\gamma}{2}\sinh z,\quad\dot{z}=\sin\frac{\gamma}{2}.\label{eq:2.7}
\end{eqnarray}
The total energy integral of the pendulum (\ref{eq:2.6}) is given
as:
\begin{equation}
E=\frac{c^{2}}{2}-\cos\gamma,\quad E\in[-1,+\infty).\label{eq:E_pend}
\end{equation}
The initial cylinder of the vertical subsystem is decomposed into
the following subsets based upon the pendulum energy that correspond
to various pendulum trajectories:
\begin{eqnarray*}
C & = & \bigcup_{i=1}^{5}C_{i},
\end{eqnarray*}
where,
\begin{eqnarray}
C_{1} & = & \left\{ \lambda\in C\,\vert\, E\in(-1,1)\right\} ,\label{eq:2.8}\\
C_{2} & = & \left\{ \lambda\in C\,\vert\, E\in(1,\infty)\right\} ,\\
C_{3} & = & \left\{ \lambda\in C\,\vert\, E=1,c\neq0\right\} ,\\
C_{4} & = & \left\{ \lambda\in C\,\vert\, E=-1,\, c=0\right\} =\left\{ (\gamma,c)\in C\,\vert\,\gamma=2\pi n,\, c=0\right\} ,\quad n\in\mathbb{N},\\
C_{5} & = & \left\{ \lambda\in C\,\vert\, E=1,\, c=0\right\} =\left\{ (\gamma,c)\in C\,\vert\,\gamma=2\pi n+\pi,\, c=0\right\} ,\quad n\in\mathbb{N}.\label{eq:2.12}
\end{eqnarray}

\begin{figure}
\centering{}\includegraphics[scale=0.5]{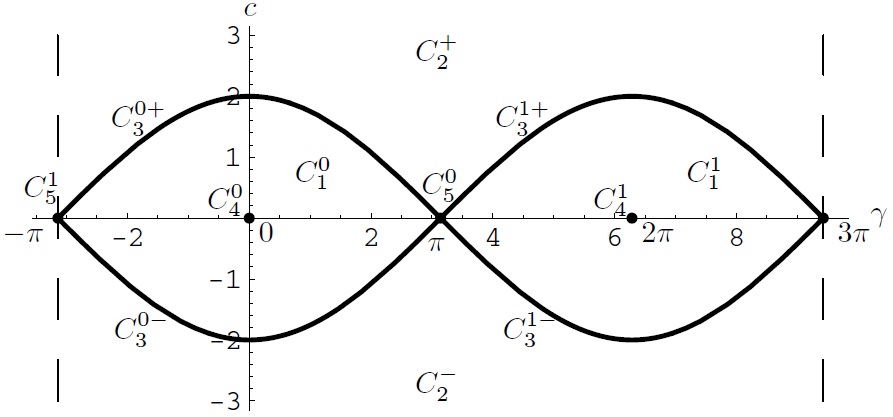}\protect\caption{\label{fig:Decomposition}Stratification of the Phase Cylinder $C$
of the Pendulum}
\end{figure}
We defined elliptic coordinates $(\varphi,k)$ for $\lambda\in\cup_{i=1}^{3}C_{i}\subset C$
and proved that the flow of the pendulum is rectified in these coordinates. Note that $k$ was defined as the reparametrized energy and $\varphi$ was defined as the reparametrized time of motion of the pendulum \cite{Extremal_Pseudo_Euclid}. Integration of the horizontal subsystem in elliptic coordinates follows from integration of the vertical subsystem and the resulting extremal trajectories are parametrized by the Jacobi elliptic functions $\mathrm{sn}(\varphi,k)$, $\mathrm{cn}(\varphi,k)$, $\mathrm{dn}(\varphi,k)$, $\mathrm{E}(\varphi,k)=\intop_{0}^{\varphi}\mathrm{dn^{2}}(t,k)dt$ (Theorems 5.1--5.5 \cite{Extremal_Pseudo_Euclid}). The results of integration for $\lambda\in C_{i},\quad i=1,\ldots,5,$ are summarized as: 

\begin{itemize}
\item Case 1 : $\lambda=(\varphi,k)\in C_{1}$
\begin{equation}
\left(\begin{array}{c}
x_{t}\\
y_{t}\\
z_{t}
\end{array}\right)=\left(\begin{array}{c}
\frac{s_{1}}{2}\left[\left(w+\frac{1}{w\left(1-k^{2}\right)}\right)\left[\mathrm{E}(\varphi_{t})-\mathrm{E}(\varphi)\right]+\left(\frac{k}{w(1-k^{2})}-kw\right)\left[\mathrm{sn}\,\varphi_{t}-\mathrm{sn}\,\varphi\right]\right]\\
\frac{1}{2}\left[\left(w-\frac{1}{w\left(1-k^{2}\right)}\right)\left[\mathrm{E}(\varphi_{t})-\mathrm{E}(\varphi)\right]-\left(\frac{k}{w\left(1-k^{2}\right)}+kw\right)\left[\mathrm{sn}\,\varphi_{t}-\mathrm{sn}\,\varphi\right]\right]\\
s_{1}\ln\left[(\mathrm{dn}\,\varphi_{t}-k\mathrm{cn}\,\varphi_{t}).w\right]
\end{array}\right),\label{eq:2.13}
\end{equation}
where $w=\frac{1}{\mathrm{dn}\varphi-k\mathrm{cn}\varphi}$, $s_{1}=\mathrm{sgn}\left(\cos\frac{\gamma}{2}\right)$
and $\varphi_{t}=\varphi+t$.
\item Case 2 : $\lambda=(\psi,k)\in C_{2}$ 
\begin{eqnarray}
x_{t} & = & \frac{1}{2}\left(\frac{1}{w(1-k^{2})}-w\right)\left[\mathrm{E}(\psi_{t})-\mathrm{E}(\psi)-k^{\prime2}\left(\psi_{t}-\psi\right)\right]\nonumber \\
 & + & \frac{1}{2}\left(kw+\frac{k}{w(1-k^{2})}\right)\left[\mathrm{sn}\,\psi_{t}-\mathrm{sn}\,\psi\right],\nonumber \\
y_{t} & = & -\frac{s_{2}}{2}\left(\frac{1}{w(1-k^{2})}+w\right)\left[\mathrm{E}(\psi_{t})-\mathrm{E}(\psi)-k^{\prime2}(\psi_{t}-\psi)\right]\nonumber \\
 & + & \frac{s_{2}}{2}\left(kw-\frac{k}{w(1-k^{2})}\right)\left[\mathrm{sn}\,\psi_{t}-\mathrm{sn}\,\psi\right],\nonumber \\
z_{t} & = & s_{2}\ln\left[\left(\mathrm{dn}\,\psi_{t}-k\mathrm{cn}\,\psi_{t}\right).w\right],\label{eq:2.14}
\end{eqnarray}
where $\psi=\frac{\varphi}{k}$, $\quad\psi_{t}=\frac{\varphi_{t}}{k}=\psi+\frac{t}{k}$
and $w=\frac{1}{\mathrm{dn}\,\psi-k\mathrm{cn}\,\psi}$, $s_{2}=\mathrm{sgn}\, c$,
$k^{\prime}=\sqrt{1-k^{2}}$. 
\item Case 3 : $\lambda=(\varphi,k)\in C_{3}$
\begin{equation}
\left(\begin{array}{c}
x_{t}\\
y_{t}\\
z_{t}
\end{array}\right)=\left(\begin{array}{c}
\frac{s_{1}}{2}\left[\frac{1}{w}\left(\varphi_{t}-\varphi\right)+w\left(\tanh\varphi_{t}-\tanh\varphi\right)\right]\\
\frac{s_{2}}{2}\left[\frac{1}{w}\left(\varphi_{t}-\varphi\right)-w\left(\tanh\varphi_{t}-\tanh\varphi\right)\right]\\
-s_{1}s_{2}\ln[w\,\textrm{sech}\,\varphi_{t}]
\end{array}\right),\label{eq:2.15}
\end{equation}
where $w=\cosh\varphi$. 
\item Case 4 : $\lambda=(\varphi,k)\in C_{4}$
\begin{equation}
\left(\begin{array}{c}
x\\
y\\
z
\end{array}\right)=\left(\begin{array}{c}
\mathrm{sgn}\left(\cos\frac{\gamma}{2}\right)t\\
0\\
0
\end{array}\right).\label{eq:2.16}
\end{equation}

\item Case 5 : $\lambda=(\varphi,k)\in C_{5}$
\begin{equation}
\left(\begin{array}{c}
x\\
y\\
z
\end{array}\right)=\left(\begin{array}{c}
0\\
0\\
\mathrm{sgn}\left(\sin\frac{\gamma}{2}\right)t
\end{array}\right).\label{eq:2.17}
\end{equation}

\end{itemize}


The phase portrait of the pendulum admits a discrete group of symmetries
$G=\{Id,\varepsilon^{1},\ldots,\varepsilon^{7}\}$. The symmetries
$\varepsilon^{i}$ are reflections and translations about the coordinates
axes $(\gamma,c)$. The reflection symmetries in the phase portrait
of a standard pendulum are given as: 
\begin{equation}
\begin{alignedat}{1}\varepsilon^{1}:(\gamma,c) & \rightarrow(\gamma,-c),\\
\varepsilon^{2}:(\gamma,c) & \rightarrow(-\gamma,c),\\
\varepsilon^{3}:(\gamma,c) & \rightarrow(-\gamma,-c),\\
\varepsilon^{4}:(\gamma,c) & \rightarrow(\gamma+2\pi,c),\\
\varepsilon^{5}:(\gamma,c) & \rightarrow(\gamma+2\pi,-c),\\
\varepsilon^{6}:(\gamma,c) & \rightarrow(-\gamma+2\pi,c),\\
\varepsilon^{7}:(\gamma,c) & \rightarrow(-\gamma+2\pi,-c).
\end{alignedat}
\label{eq:symm}
\end{equation}
According to Proposition 6.3 \cite{Extremal_Pseudo_Euclid}, the action
of reflections on endpoints of extremal trajectories can be defined
as $\varepsilon^{i}:q\mapsto q^{i}$, where $q=(x,y,z)\in M,\quad q^{i}=(x^{i},y^{i},z^{i})\in M$
and, 
\begin{align}
(x^{1},y^{1},z^{1}) & =(x\cosh z-y\sinh z,\, x\sinh z-y\cosh z,\, z),\nonumber \\
(x^{2},y^{2},z^{2}) & =(x\cosh z-y\sinh z,\,-x\sinh z+y\cosh z,\,-z),\nonumber \\
(x^{3},y^{3},z^{3}) & =(x,\,-y,\,-z),\nonumber \\
(x^{4},y^{4},z^{4}) & =(-x,\, y,\,-z),\label{eq:symm_M}\\
(x^{5},y^{5},z^{5}) & =(-x\cosh z+y\sinh z,\, x\sinh z-y\cosh z,\,-z),\nonumber \\
(x^{6},y^{6},z^{6}) & =(-x\cosh z+y\sinh z,\,-x\sinh z+y\cosh z,\, z),\nonumber \\
(x^{7},y^{7},z^{7}) & =(-x,\,-y,\, z).\nonumber 
\end{align}
These symmetries are exploited to state the general conditions on
Maxwell strata in terms of the functions $z_{t}$ and $R_{i}(q)$
given as: 
\begin{equation}
R_{1}=y\cosh\frac{z}{2}-x\sinh\frac{z}{2},\quad R_{2}=x\cosh\frac{z}{2}-y\sinh\frac{z}{2}.\label{eq:2.18}
\end{equation}
We define the Maxwell sets $MAX^{i},\quad i=1,\ldots,7$, resulting
from the reflections $\varepsilon^{i}$ of the extremals in the preimage
of the exponential mapping $N$ as:
\[
\mathrm{MAX}^{i}=\left\{ \text{\ensuremath{\nu}}=(\text{\ensuremath{\lambda}},t)\text{\ensuremath{\in}}N=C\times\mathbb{R}^{+}\quad|\quad\lambda\neq\lambda{}^{i},\quad\mathrm{Exp}(\lambda,t)=\mathrm{Exp}(\lambda^{i},t)\right\} ,
\]
where $\lambda=\varepsilon^{i}(\lambda).$ The corresponding Maxwell
strata in the image of the exponential mapping are defined as:

\[
\mathrm{Max}^{i}=\mathrm{Exp}(\mathrm{MAX}^{i})\subset M.
\]
 In \cite{Max_Conj_SH2} Proposition 3.7 we proved that the first
Maxwell points corresponding to the reflection symmetries of the vertical
subsystem lie on the plane $z=0$ and the corresponding Maxwell time
$t_{1}^{\mathrm{Max}}(\lambda)$ is given as :
\begin{eqnarray}
\lambda\in C_{1} & \implies & t_{1}^{\mathrm{Max}}(\lambda)=4K(k),\label{eq:2.20}\\
\lambda\in C_{2} & \implies & t_{1}^{\mathrm{Max}}(\lambda)=4kK(k),\label{eq:2.23}\\
\lambda\in C_{3}\cup C_{4}\cup C_{5} & \implies & t_{1}^{\mathrm{Max}}(\lambda)=+\infty.\label{eq:2.24}
\end{eqnarray}
Similarly we proved that the first conjugate time $t_{1}^{\mathrm{conj}}(\lambda)$
is bounded as (Theorems 4.1--4.3) \cite{Max_Conj_SH2}:
\begin{eqnarray}
\lambda\in C_{1} & \implies & 4K(k)\leq t_{1}^{\mathrm{conj}}(\lambda)\leq2p_{1}^{1}(k),\label{eq:2.25}\\
\lambda\in C_{2} & \implies & 4kK(k)\leq t_{1}^{\mathrm{conj}}(\lambda)\leq2k\, p_{1}^{1}(k),\label{eq:2.26}\\
\lambda\in C_{4} & \implies & t_{1}^{\mathrm{conj}}(\lambda)=2\pi,\label{eq:2.27}\\
\lambda\in C_{3}\cup C_{5} & \implies & t_{1}^{\mathrm{conj}}(\lambda)=+\infty.\label{eq:2.28}
\end{eqnarray}
where $p_{1}^{1}(k)$ is the first positive root of the function $f_{1}(p)=\mathrm{cn}p\,\mathrm{E}(p)-\mathrm{sn}p\,\mathrm{dn}p,$
which is bounded as $p_{1}^{1}(k)\in(2K(k),3K(k))$. Note that we defined: 
\begin{eqnarray}
\varphi_{t}=\tau+p,\quad\varphi=\tau-p\implies\tau & = & \frac{1}{2}\left(\varphi_{t}+\varphi\right),\, p=\frac{t}{2}\textrm{ when }\nu=(\lambda,t)\in N_{1}\cup N_{3},\label{eq:2.29}\\
\psi_{t}=\frac{\varphi_{t}}{k}=\tau+p,\quad\psi=\frac{\varphi}{k}=\tau-p\implies\tau & = & \frac{1}{2k}\left(\varphi_{t}+\varphi\right),\, p=\frac{t}{2k}\textrm{ when }\nu=(\lambda,t)\in N_{2}.\label{eq:2.30}
\end{eqnarray}
Here and below $N_{i}=C_{i}\times\mathbb{R}_{+}$.
\section{Upper Bound on Cut Time}

In this section we describe the basic properties of the upper bound
on cut time obtained in \cite{Max_Conj_SH2}.

Define the following function $\mathbf{t}:C\rightarrow(0,+\infty]$,
\[
\mathbf{t}(\lambda)=\min\left(t_{1}^{\mathrm{Max}}(\lambda),t_{1}^{\mathrm{conj}}(\lambda)\right),\quad\lambda\in C.
\]
Equalities (\ref{eq:2.20})--(\ref{eq:2.28}) yield the explicit representation
of this function:
\begin{eqnarray}
\lambda\in C_{1} & \implies & \mathbf{t}(\lambda)=4K(k),\label{eq:ttC1}\\
\lambda\in C_{2} & \implies & \mathbf{t}(\lambda)=4kK(k),\label{eq:ttC2}\\
\lambda\in C_{4} & \implies & \mathbf{t}(\lambda)=2\pi,\label{eq:ttC4}\\
\lambda\in C_{3}\cup C_{5} & \implies & \mathbf{t}(\lambda)=+\infty.\label{eq:ttC35}
\end{eqnarray}
In \cite{Max_Conj_SH2} we proved the upper bound:
\begin{equation}
t_{\mathrm{cut}}(\lambda)\leq\mathbf{t}(\lambda),\quad\lambda\in C.\label{eq:tCutbound}
\end{equation}

We now prove that inequality (\ref{eq:tCutbound}) is in fact an equality
(see Theorem \ref{thm:cut_time_exact}). The general scheme of the
proof is as follows \cite{cut_sre1}, \cite{cut_engel}:

\begin{enumerate}
\item The exponential mapping $\mathrm{Exp}:N=C\times\mathbb{R}_{+}\rightarrow M$
parametrizes all optimal geodesics, but also all non-optimal ones,
since all the geodesics $\mathrm{Exp}(\lambda,t)$ with $t>\mathbf{t}(\lambda)$
are not optimal. 
\item We reduce the domain of the exponential mapping so that it does not
include these a priori non-optimal geodesics:
\[
\widehat{N}=\left\{ (\lambda,t)\in N\quad\mid\quad t\leq\mathbf{t}(\lambda)\right\} .
\]
We also reduce the range of the exponential mapping so that it does
not contain the initial point for which the optimal geodesic is trivial:
\[
\widehat{M}=M\backslash\left\{ q_{0}\right\} .
\]
Then $\mathrm{Exp}:\widehat{N}\rightarrow\widehat{M}$ is surjective,
but not injective, due to Maxwell points. 
\item We exclude Maxwell points in the image of $\mathrm{Exp}$:
\[
\widetilde{M}=\left\{ q\in M\quad\mid\quad\varepsilon^{i}(q)\neq q\right\} ,
\]
and reduce respectively the preimage of $\mathrm{Exp}$:
\[
\widetilde{N}=\mathrm{Exp}^{-1}\left(\widetilde{M}\right).
\]
The mapping $\mathrm{Exp}:\widetilde{N}\rightarrow\widetilde{M}$
is injective. Moreover, it is non-degenerate since $t_{1}^{\mathrm{conj}}(\lambda)\geq\mathbf{t}(\lambda)$.
\item We take connected components in preimage and image of $\mathrm{Exp}:$
\[
\widetilde{N}=\cup D_{i},\qquad\widetilde{M}=\cup M_{i}.
\]
Each of the mappings $\mathrm{Exp}:D_{i}\rightarrow M_{i}$ is non-degenerate
and proper. Moreover, all $D_{i}$ and $M_{i}$ are smooth 3-dimensional
manifolds, connected and simply connected. By Hadamard's global diffeomorphism
theorem \cite{diffeo}, each $\mathrm{Exp}:D_{i}\rightarrow M_{i}$
is a diffeomorphism. Thus $\mathrm{Exp}:\widetilde{N}\rightarrow\widetilde{M}$
is a diffeomorphism as well.
\item Further, we consider the action of the exponential mapping on the
boundary of the 3-dimensional diffeomorphic domains:
\begin{eqnarray*}
\mathrm{Exp}:N^{\prime} & \rightarrow & M^{\prime},\quad N^{\prime}=\widehat{N}\backslash\widetilde{N},\qquad M^{\prime}=\widehat{M}\backslash\widetilde{M}.
\end{eqnarray*}
We construct a stratification in the preimage and the image of $\mathrm{Exp}:$
\begin{eqnarray*}
N^{\prime} & = & \cup N_{i}^{\prime},\quad M^{\prime}=\cup M_{i}^{\prime},\\
\textrm{dim }N_{i}^{\prime},\textrm{\,\ dim }M_{i}^{\prime} & \in & \left\{ 0,1,2\right\} ,
\end{eqnarray*}
where all $N_{i}^{\prime}$ are disjoint, while some $M_{i}^{\prime}$
coincide with others. Further, we prove that all $\mathrm{Exp}:N_{i}^{\prime}\rightarrow M_{i}^{\prime}$
are diffeomorphisms by the same argument. 
\item On the basis of the global diffeomorphic structure of the exponential
mapping thus described, we get the following results:
\begin{eqnarray*}
t_{\mathrm{cut}}(\lambda) & = & \mathbf{t}(\lambda),\quad\lambda\in C,\\
\mathrm{Max} & = & \cup\left\{ M_{i}^{\prime}\quad\mid\quad\exists\, j\neq i\textrm{ such that }M_{j}^{\prime}=M_{i}^{\prime}\right\} ,\\
\mathrm{Cut} & = & \mathrm{cl}(\mathrm{Max})\backslash\left\{ q_{0}\right\} ,\\
\mathrm{Cut}\cap\mathrm{Conj} & = & \partial(\mathrm{Max})\backslash\left\{ q_{0}\right\} .
\end{eqnarray*}
We show that the optimal synthesis is double valued on the Maxwell
set $\mathrm{Max}$, and is one valued on $\widehat{M}\backslash\mathrm{Max}$.\\
The central notion of our approach is the stratification in the preimage
and in the image of $\mathrm{Exp}:$
\begin{alignat*}{2}
\widehat{N} & =\left(\cup D_{i}\right)\cup\left(\cup N_{i}^{\prime}\right),\\
\widehat{M} & =\left(\cup M_{i}\right)\cup\left(\cup M_{i}^{\prime}\right),\\
\textrm{dim}(D_{i}) & =\textrm{dim}(M_{i})=3,\\
\textrm{dim}(N_{i}^{\prime}),\textrm{\,\ dim}(M_{i}^{\prime}) & \in\left\{ 0,1,2\right\} , & {}
\end{alignat*}
such that all the corresponding strata are diffeomorphic via the exponential
mapping, i.e., $\mathrm{Exp}:D_{i}\rightarrow M_{i}$ and $\mathrm{Exp}:N_{i}^{\prime}\rightarrow M_{i}^{\prime}$
are diffeomorphisms. 
\end{enumerate}


It is well known \cite{cut_engel},\cite{diffeo} that for any
smooth manifolds $X$ and $Y$ of equal dimensions, a smooth mapping
$f:X\rightarrow Y$ is a diffeomorphism if $f$, $X$ and $Y$ satisfy
the following conditions \textbf{P1 }-- \textbf{P4}:

\textbf{P1} - $X$ is connected, 

\textbf{P2 }- $Y$ is connected and simply connected, 

\textbf{P3 }-\textbf{ }$f$ is non-degenerate, 

\textbf{P4 }- $f$ is proper, i.e., for any compact set $K\subset Y$
the inverse image $f^{-1}(K)\subset X$ is also compact. 

We now consider the invariance properties of the function $\mathbf{t}$
with respect to the reflections $\varepsilon^{i}\in G$ and the vertical
part of the Hamiltonian vector field:
\[
\overrightarrow{H}_{\nu}=c\frac{\partial}{\partial\gamma}-\sin\gamma\frac{\partial}{\partial c}\in\mathrm{Vec}(C).
\]
\begin{proposition}
\label{prop:tt_invar} $\qquad$\\

\text{\emph{(1)}}$\quad$The function $\mathbf{t}$ is invariant
w.r.t. the reflections $\varepsilon^{i}\in G$ and the flow of $\overrightarrow{H}_{\nu}$:
\[
\mathbf{t}\circ\varepsilon^{i}(\lambda)=\mathbf{t}\circ e^{t\overrightarrow{H}_{\nu}}(\lambda)=\mathbf{t}(\lambda),\quad\lambda\in C,\quad\varepsilon^{i}\in G,\quad t\in\mathbb{R}.
\]

\text{\emph{(2)}}$\quad$The function $\mathbf{t}:C\rightarrow(0,+\infty]$
is in fact a function $\mathbf{t}(E)$ of the energy $E=\frac{c^{2}}{2}-\cos\gamma$
of pendulum \text{\emph{(\ref{eq:2.6})}}. \end{proposition}
\begin{proof}
The reflections $\varepsilon^{i}\in G$ (\ref{eq:symm}) and the flow
of $\overrightarrow{H}_{\nu}$ preserve the subsets $C_{i}$ of the
cylinder $C$ and on each of these subsets, the function $\mathbf{t}$
is expressed as a function of the energy $E$ of the pendulum since
we have equalities (\ref{eq:ttC1})--(\ref{eq:ttC35}) and,
\begin{eqnarray*}
\lambda\in C_{1} & \implies & k=\sqrt{\frac{E+1}{2}},\\
\lambda\in C_{2} & \implies & k=\sqrt{\frac{2}{E+1}},\\
\lambda\in C_{4} & \implies & E=-1,\\
\lambda\in C_{3}\cup C_{5} & \implies & E=1.
\end{eqnarray*}
This proves item (2) of this proposition. Item (1) follows since the
energy $E$ is invariant w.r.t. $\varepsilon^{i}$ and $\overrightarrow{H}_{\nu}$.
\hfill$\square$ 
\end{proof}

A plot of $\mathbf{t}(E)$ is shown in Figure \ref{fig:tt}. Regularity
properties of the function $\mathbf{t}(E)$ visible in its plot are proved in the following statement.
\begin{figure}
\begin{centering}
\includegraphics{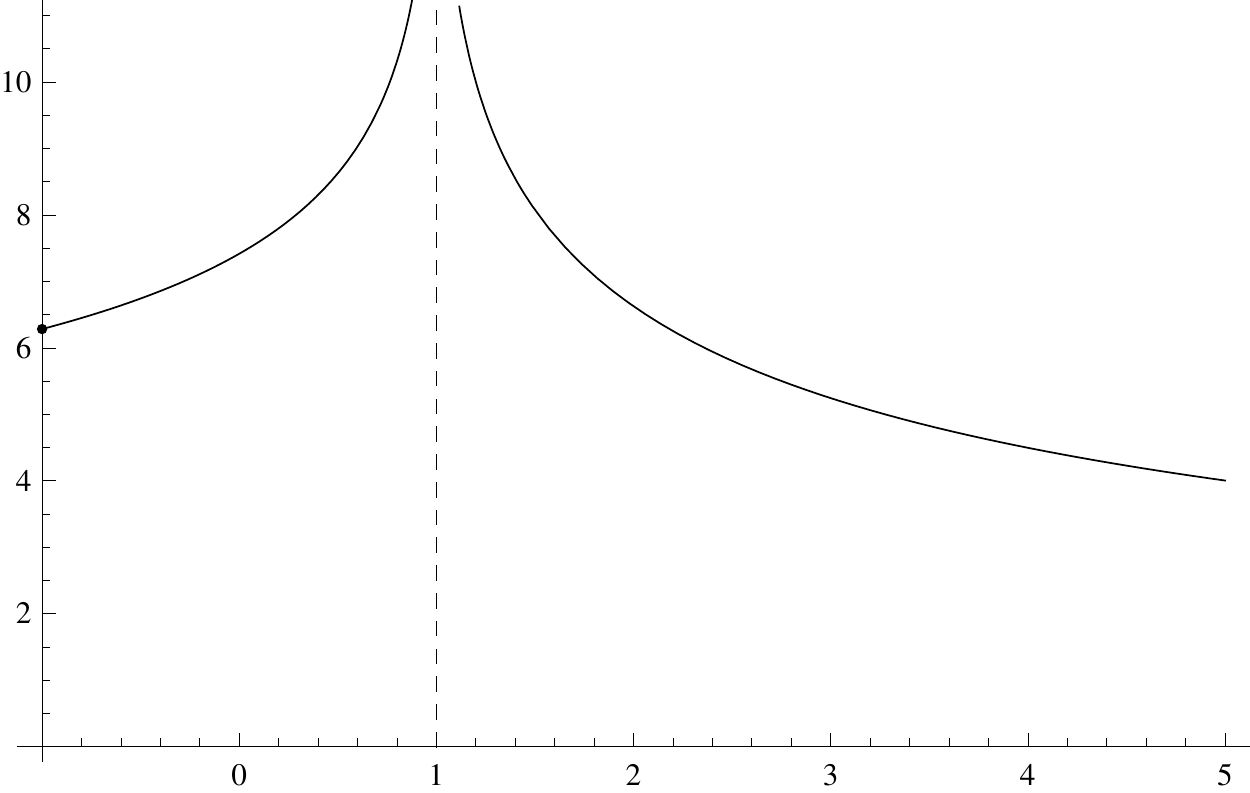}
\par\end{centering}

\protect\caption{\label{fig:tt}Plot of the function $\mathbf{t}(E)$}
\end{figure}

\begin{proposition}
\label{prop:tt-reg}$\qquad$\\

\text{\emph{(1)}}$\quad$The function $\mathbf{t}(\lambda)$ is smooth
on $C_{1}\cup C_{2}$. 

\text{\emph{(2)}}$\quad$$\lim_{E\rightarrow-1}\mathbf{t}(E)=2\pi,\quad\lim_{E\rightarrow1}\mathbf{t}(E)=+\infty,\quad\lim_{E\rightarrow+\infty}\mathbf{t}(E)=0$.

\text{\emph{(3)}}$\quad$The function $\mathbf{t}:C\rightarrow(0,+\infty]$
is continuous.\end{proposition}
\begin{proof}
Item (1) follows from (\ref{eq:ttC1}) and (\ref{eq:ttC2}). The limits
in item (2) follow from (\ref{eq:ttC1}) and (\ref{eq:ttC2}), and
from the limits $\lim_{k\rightarrow+0}K(k)=\frac{\pi}{2},\quad\lim_{k\rightarrow1-}K(k)=+\infty$.
Then continuity of $\mathbf{t}(\lambda)$ follows on $C_{4}$:
\[
\lambda\rightarrow\bar{\lambda}\in C_{4}\implies E(\lambda)\rightarrow E(\bar{\lambda})=-1\implies\mathbf{t}(\lambda)\rightarrow2\pi=\mathbf{t}(\bar{\lambda}).
\]
Continuity on $C_{3}\cup C_{5}$ follows since
\[
\lambda\rightarrow\bar{\lambda}\in C_{3}\cup C_{5}\implies E(\lambda)\rightarrow E(\bar{\lambda})=1\implies\mathbf{t}(\lambda)\rightarrow+\infty=\mathbf{t}(\bar{\lambda}).
\]
Thus $\mathbf{t}(\lambda)$ is continuous on $C$ and item (3) is
proved. \hfill$\square$
\end{proof}

\subsection{Decompositions in the Image of the Exponential Mapping}

Consider the set $\widehat{M}=M\backslash\{q_{0}\}$. From Filippov's
theorem and Pontryagin's Maximum Principle \cite{agrachev_sachkov},
we already know that any point $q\in\widehat{M}$ can be joined with
$q_{0}$ by an optimal trajectory $q(s)=\mathrm{Exp}(\lambda,s)$
such that $q(t)=q,\quad(\lambda,t)\in N$. Then $\mathrm{Exp}(N)\supset\widehat{M}$.
However the Maxwell points $q\in\widehat{M}$ have non unique preimage
under the exponential mapping. Hence the mapping $\mathrm{Exp}:N\rightarrow\widehat{M}$
is surjective, but not injective. In order to separate Maxwell points
we consider the set that contains all such points:
\[
M^{\prime}=\left\{ q\in M\quad\mid\quad z=0,\quad x^{2}+y^{2}\neq0\right\} ,
\]
and its complement $\widetilde{M}$ in $\widehat{M}$:
\begin{align*}
\widetilde{M} & =\left\{ q\in M\quad\mid\quad z\neq0\right\} ,\\
\widehat{M} & =\widetilde{M}\sqcup M^{\prime},
\end{align*}
where $\sqcup$ is the union of disjoint sets. 

\subsubsection{Decompositions in $\widetilde{M}$}

The plane $z=0$ cuts the domain $\widetilde{M}$ into two half spaces
as: 
\begin{eqnarray}
\widetilde{M} & = & M_{1}\sqcup M_{2},\nonumber \\
M_{1} & = & \left\{ q\in M\quad\mid\quad z>0\right\} ,\label{eq:M1}\\
M_{2} & = & \left\{ q\in M\quad\mid\quad z<0\right\} .\label{eq:M2}
\end{eqnarray}
Note that the decomposition of the manifold $M$ is simpler in description
of cut time on $\mathrm{SH}(2)$ than similar decomposition of $M$
in related problems on $\mathrm{SE(2)}$ \cite{cut_sre1} and on the
Engel group \cite{cut_engel}. 
\begin{proposition}
\label{prop:Ref_Mi}Reflections $\varepsilon^{j}\in G$ permute the
domains $M_{1}$ and $M_{2}$ according to Table \text{\emph{\ref{tab:1}}}.
\begin{table}
\centering{}%
\begin{tabular}{|c|c|}
\hline 
$\mathrm{Id}$,$\varepsilon^{1}$,$\varepsilon^{6}$,$\varepsilon^{7}$ & $\varepsilon^{2}$,$\varepsilon^{3}$,$\varepsilon^{4}$,$\varepsilon^{5}$\tabularnewline
\hline 
$M_{1}$ & $M_{2}$\tabularnewline
\hline 
$M_{2}$ & $M_{1}$\tabularnewline
\hline 
\end{tabular}\protect\caption{\label{tab:1}Action of $\varepsilon^{i}$ on $M_{j}$}
\end{table}
\end{proposition}

\begin{proof}
Follows immediately from the definitions of the actions of reflections
(\ref{eq:symm_M}).\hfill$\square$\end{proof}

\begin{proposition}
\label{prop:Mi-topol}The domains $M_{1},M_{2}$ are open, connected
and simply connected. \end{proposition}
\begin{proof}
From the definition of the sets $M_{1}$, $M_{2}$ (\ref{eq:M1})--(\ref{eq:M2})
it follows that the domains $M_{i}$ are homeomorphic to $\mathbb{R}^{3}$
and therefore they are open, connected and simply connected. \hfill$\square$
\end{proof}
\subsection{Decomposition in the Preimage of the Exponential Mapping}

We now consider the following set $\widehat{N}\subset N$ corresponding
to all potentially optimal geodesics: 
\[
\widehat{N}=\left\{ (\lambda,t)\in N\quad\mid\quad t\leq\mathbf{t}(\lambda)\right\} .
\]
By existence of the optimal geodesics, $\mathrm{Exp}(\widehat{N})\supset\widehat{M}$.
In order to separate the Maxwell points in the preimage of the exponential
mapping, introduce further the sets:
\begin{align*}
\widehat{N} & =\widetilde{N}\sqcup N^{\prime},\\
N^{\prime} & =\left\{ (\lambda,t)\in\cup{}_{i=1}^{3}\widehat{N}_{i}\quad\mid\quad t=\mathbf{t}(\lambda)\textrm{ or }\sin\frac{\gamma_{t/2}}{2}=0\right\} \cup\widehat{N}_{4},\\
\widehat{N}_{i} & =N_{i}\cap\widehat{N},\quad i=1,\ldots,4,\\
\widetilde{N} & =\left\{ (\lambda,t)\in\cup_{i=1}^{3}N_{i}\quad\mid\quad t<\mathbf{t}(\lambda),\quad\sin\frac{\gamma_{t/2}}{2}\neq0\right\} \cup N_{5}.
\end{align*}

\subsubsection{Decomposition in $\widetilde{N}$}

We now introduce the connected components $D_{i}$ of the set $\widetilde{N}$:
\begin{eqnarray*}
\widetilde{N} & = & D_{1}\sqcup D_{2},\\
D_{1} & = & \left\{ (\lambda,t)\in\cup_{i=1}^{3}N_{i}\quad\mid\quad t<\mathbf{t}(\lambda),\quad\sin\left(\frac{\gamma_{t/2}}{2}\right)>0\right\} ,\\
D_{2} & = & \left\{ (\lambda,t)\in\cup_{i=1}^{3}N_{i}\quad\mid\quad t<\mathbf{t}(\lambda),\quad\sin\left(\frac{\gamma_{t/2}}{2}\right)<0\right\} ,
\end{eqnarray*}
where $D_{i}$ are defined explicitly in coordinates in Table \ref{tab:2}
(in the sets $N_{1},N_{2},N_{3}$). Projections of the sets $D_{i}$
to the initial phase cylinder are shown in Figure \ref{fig:Projs_Di}.
We note that for $t<\mathbf{t}(\lambda)=t_{1}^{\mathrm{Max}}(\lambda)$
the values of $p$ are given from formulas (\ref{eq:2.29})--(\ref{eq:2.30}),
and the values of $t_{1}^{\mathrm{Max}}(\lambda)$ are given in (\ref{eq:2.20})--(\ref{eq:2.24}).
The values of $\tau$ in Table \ref{tab:2} were calculated by using
the definition of elliptic coordinates \cite{Extremal_Pseudo_Euclid},
formulas for Jacobi elliptic functions \cite{Table_Int} and values
of $\gamma$ and $c$ from Figure \ref{fig:Decomposition}. Note that
enumeration of the sets $D_{i}$ is chosen to correspond to the sets
$M_{i}$ for further analysis.

\begin{figure}
\begin{centering}
\includegraphics[scale=0.6]{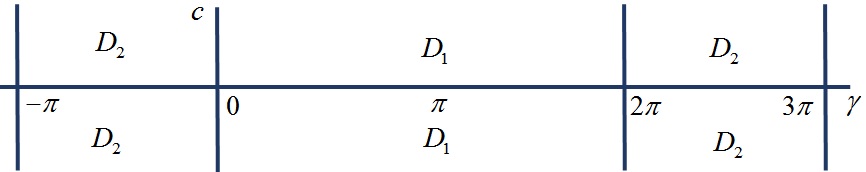}\protect\caption{\label{fig:Projs_Di}Projections of $D_{i}$ to Phase Cylinder $C$
of the Pendulum at $t=0$}

\par\end{centering}

\end{figure}
\begin{table}
\centering{}%
\begin{tabular}{|c||c||c||c||c|}
\hline 
$D_{i}$ & \multicolumn{2}{c||}{$D_{1}$} & \multicolumn{2}{c|}{$D_{2}$}\tabularnewline
\hline 
\hline 
$\begin{array}{c}
\lambda\\
p\\
\tau
\end{array}$ & $\begin{array}{c}
C_{1}^{0}\\
(0,2K)\\
(0,2K)
\end{array}$ & $\begin{array}{c}
C_{1}^{1}\\
(0,2K)\\
(2K,4K)
\end{array}$ & $\begin{array}{c}
C_{1}^{0}\\
(0,2K)\\
(2K,4K)
\end{array}$ & $\begin{array}{c}
C_{1}^{1}\\
(0,2K)\\
(0,2K)
\end{array}$\tabularnewline
\hline 
\hline 
$\begin{array}{c}
\lambda\\
p\\
\tau
\end{array}$ & $\begin{array}{c}
C_{2}^{+}\\
(0,2K)\\
(0,2K)
\end{array}$ & $\begin{array}{c}
C_{2}^{-}\\
(0,2K)\\
(-2K,0)
\end{array}$ & $\begin{array}{c}
C_{2}^{+}\\
(0,2K)\\
(2K,4K)
\end{array}$ & $\begin{array}{c}
C_{2}^{-}\\
(0,2K)\\
(0,2K)
\end{array}$\tabularnewline
\hline 
\hline 
$\begin{array}{c}
\lambda\\
p\\
\tau
\end{array}$ & $\begin{array}{c}
C_{3}^{0+}\cup C_{3}^{1-}\\
(0,+\infty)\\
(0,+\infty)
\end{array}$ & $\begin{array}{c}
C_{3}^{0-}\cup C_{3}^{1+}\\
(0,+\infty)\\
(-\infty,0)
\end{array}$ & $\begin{array}{c}
C_{3}^{0+}\cup C_{3}^{1-}\\
(0,+\infty)\\
(-\infty,0)
\end{array}$ & $\begin{array}{c}
C_{3}^{0-}\cup C_{3}^{1+}\\
(0,+\infty)\\
(0,+\infty)
\end{array}$\tabularnewline
\hline 
\end{tabular}\protect\caption{\label{tab:2}Decomposition $\widetilde{N}=\cup_{i=1}^{2}D_{i}$}
\end{table}
We now establish an important fact about the domains $D_{i}$ that is vital in proving that the exponential mapping transforms $D_{i}$ diffeomorphically.
\begin{proposition}
\label{prop:Ref_Di}Reflections $\varepsilon^{j}\in G$ permute the domains $D_{1}$ and $D_{2}$ as shown in Table \text{\emph{\ref{tab:3}}}.
\begin{table}
\centering{}%
\begin{tabular}{|c|c|}
\hline 
$\mathrm{Id}$,$\varepsilon^{1}$,$\varepsilon^{6}$,$\varepsilon^{7}$ & $\varepsilon^{2}$,$\varepsilon^{3}$,$\varepsilon^{4}$,$\varepsilon^{5}$\tabularnewline
\hline 
$D_{1}$ & $D_{2}$\tabularnewline
\hline 
$D_{2}$ & $D_{1}$\tabularnewline
\hline 
\end{tabular}\protect\caption{\label{tab:3}Action of $\varepsilon^{i}$ on $D_{j}\subset\widetilde{N}$}
\end{table}
\end{proposition}
\begin{proof}
In paper \cite{Extremal_Pseudo_Euclid} we defined the action of reflections
$\varepsilon^{j}:N\rightarrow N$ so that it satisfies the following
properties:
\begin{eqnarray*}
\varepsilon^{j}(\lambda,t) & = & \left(\varepsilon^{j}\circ e^{t\overrightarrow{H}_{\nu}}(\lambda),t\right),\quad\textrm{if \quad}\varepsilon_{*}^{j}\overrightarrow{H}_{\nu}=-\overrightarrow{H}_{\nu},\\
\varepsilon^{j}(\lambda,t) & = & \left(\varepsilon^{j}(\lambda),t\right),\quad\textrm{if \quad}\varepsilon_{*}^{j}\overrightarrow{H}_{\nu}=\overrightarrow{H}_{\nu},
\end{eqnarray*}
where $\varepsilon_{*}^{j}\left(\overrightarrow{H}_{\nu}\right)$
is the pushforward of $\overrightarrow{H}_{\nu}$ under the reflection
$\varepsilon^{j}$. Recall that $\varepsilon_{*}^{j}\overrightarrow{H}_{\nu}=-\overrightarrow{H}_{\nu},\textrm{ for }j=1,2,5,6$
because these symmetries reverse the direction of time and $\varepsilon_{*}^{j}\overrightarrow{H}_{\nu}=\overrightarrow{H}_{\nu},\textrm{ for }j=3,4,7$
because these symmetries preserve the direction of time \cite{Extremal_Pseudo_Euclid}.
Hence, it is sufficient to prove the case $\varepsilon^{2}(D_{1})=D_{2}$
as proof of all other cases $\varepsilon^{j}(D_{i})=D_{k}$ is similar.
In order to prove the inclusion $\varepsilon^{j}(D_{1})\subset D_{2}$
we take any $(\lambda,t)=(\gamma,c,t)\in D_{1}$ and prove that
\[
\varepsilon^{2}:(\lambda,t)\mapsto(\lambda^{2},t)=(\gamma^{2},c^{2},t)\in D_{2}.
\]
By Proposition \ref{prop:tt_invar},
\[
\mathbf{t}(\lambda^{2})=\mathbf{t}\circ\varepsilon^{2}\circ e^{t\overrightarrow{H}_{\nu}}(\lambda)=\mathbf{t}(\lambda).
\]
Thus $t<\mathbf{t}(\lambda)$. Moreover, at instant $t/2$ the trajectories
of the vertical subsystem are given as:
\begin{eqnarray*}
\lambda_{t/2} & = & (\gamma_{t/2},c_{t/2})=e^{\overrightarrow{H}_{\nu}t/2}(\lambda),\\
\lambda_{t/2}^{2} & = & \left(\gamma_{t/2}^{2},c_{t/2}^{2}\right)=e^{\overrightarrow{H}_{\nu}t/2}\left(\lambda^{2}\right),
\end{eqnarray*}
Since $\lambda^{2}=\varepsilon^{2}\circ e^{\overrightarrow{H}_{\nu}t}(\lambda)$,
we have
\begin{eqnarray}
\lambda_{t/2}^{2} & = & e^{\overrightarrow{H}_{\nu}t/2}\circ\varepsilon^{2}\circ e^{\overrightarrow{H}_{\nu}t}(\lambda)=\varepsilon^{2}\circ e^{-\overrightarrow{H}_{\nu}t/2}\circ e^{\overrightarrow{H}_{\nu}t}(\lambda)=\varepsilon^{2}\circ e^{\overrightarrow{H}_{\nu}t/2}(\lambda)=\varepsilon^{2}(\lambda_{t/2}).\label{eq:3.8}
\end{eqnarray}
In proof of (\ref{eq:3.8}) we used the fact that for any diffeomorphism
$F:M\rightarrow M$ and a vector field $\overrightarrow{V}$ on a
manifold $M$, $F_{*}\overrightarrow{V}=-\overrightarrow{V}\Longleftrightarrow F\circ e^{t\overrightarrow{V}}=e^{-t\overrightarrow{V}}\circ F$.
Clearly, $\varepsilon^{2}(\lambda_{t/2})=\left(\gamma_{t/2}^{2},c_{t/2}^{2}\right)$
and from (6.3) \cite{Extremal_Pseudo_Euclid} we have:
\[
\left(\gamma_{t/2}^{2},c_{t/2}^{2}\right)=\left(-\gamma_{t/2},c_{t/2}\right).
\]
Thus $\sin\frac{\gamma_{t/2}^{2}}{2}=\sin\frac{-\gamma_{t/2}}{2}<0$.
We proved that $(\lambda^{2},t)\in D_{2}$, thus $\varepsilon^{2}(D_{1})\subset D_{2}$.
 Similarly it follows that $\varepsilon^{2}(D_{2})\subset D_{1}$.
Since $\varepsilon^{2}\circ\varepsilon^{2}=\mathrm{Id}$, then $\varepsilon^{2}\circ\varepsilon^{2}(D_{1})=D_{1}\implies\varepsilon^{2}(D_{1})=D_{2}$.
\hfill$\square$\end{proof}

\begin{proposition}
\label{prop:Di-topol}The domains $D_{1},D_{2}\subset\widetilde{N}$
are open and connected. \end{proposition}
\begin{proof}
Since $\varepsilon^{2}:N\rightarrow N$ is a diffeomorphism and $\varepsilon^{2}(D_{1})=D_{2}$
it suffices to prove that $D_{1}$ is open and connected. Consider
a vector field 
\[
P=\frac{t}{2}\left(c\frac{\partial}{\partial\gamma}-\sin\gamma\frac{\partial}{\partial c}\right)\in\mathrm{Vec}(N).
\]
The flow of this vector field $e^{P}$ is given as:
\[
e^{P}(\gamma,c,t)=e^{P}(\lambda,t)=\left(e^{\frac{t}{2}\overrightarrow{H}_{\nu}}(\lambda),t\right)=\left(\gamma_{t/2},c_{t/2},t\right).
\]
Thus $e^{P}(D_{1})=\widetilde{D}_{1}$ where 
\[
\widetilde{D}_{1}=\left\{ (\lambda,t)\in N\quad\mid\quad\sin\frac{\gamma}{2}>0,\quad t<\mathbf{t}(\lambda)\right\} .
\]
The set $\widetilde{D}_{1}$ is a subgraph of a continuous function
$\lambda\mapsto\mathbf{t}(\lambda)$ on an open connected 2-dimensional
domain $\left\{ (\gamma,c)\in C\quad\mid\quad\gamma\in(0,2\pi),\quad c\in\mathbb{R}\right\} $,
thus $\widetilde{D}_{1}$ is open and connected. Since $D_{1}=e^{-P}(\widetilde{D}_{1})$
therefore $D_{1}$ is also open and connected.  \hfill$\square$\end{proof}

\begin{proposition}
\label{prop:Exp Di Mi}There hold the inclusions:\end{proposition}
\begin{enumerate}
\item[(1)] $\mathrm{Exp}(D_{i})\subset M_{i},\quad i=1,2,$
\item[(2)] $\mathrm{Exp}(\widetilde{N})\subset\widetilde{M},$
\item[(3)] $\mathrm{Exp}(N^{\prime})\subset M^{\prime}.$\end{enumerate}
\begin{proof}
$\qquad$\end{proof}

\begin{enumerate}
\item[(1)] It suffices to prove only that $\mathrm{Exp}(D_{1})\subset M_{1}$,
in view of the reflections $\varepsilon^{j}$. Notice the decomposition:
\begin{equation}
D_{1}=\left(D_{1}\cap N_{1}\right)\sqcup\left(D_{1}\cap N_{2}\right)\sqcup\left(D_{1}\cap N_{3}\right)\sqcup\left(D_{1}\cap N_{5}\right).\label{eq:D1decomp}
\end{equation}
Let $(\lambda,t)\in D_{1}\cap N_{1}=\left\{ (\lambda,t)\in N_{1}\quad\mid\quad t<\mathbf{t}(\lambda),\quad\sin\frac{\gamma_{t/2}}{2}>0\right\} ,$
thus $p=\frac{t}{2}\in(0,2K(k))$. Further, from formula (5.3) \cite{Extremal_Pseudo_Euclid}
we have $s_{1}\mathrm{sn}\tau>0$. Now recall formula (3.2) \cite{Max_Conj_SH2}:
\begin{equation}
\sinh z_{t}=s_{1}\frac{2k\,\mathrm{sn}p\:\mathrm{sn}\tau}{\Delta},\quad\Delta=1-k^{2}\mathrm{sn^{2}}p\,\mathrm{sn^{2}}\tau.\label{eq:3.10}
\end{equation}
Then we get $\sinh z_{t}>0$, thus $z_{t}>0$, i.e., $\mathrm{Exp}(\lambda,t)\in M_{1}$.
We proved that $\mathrm{Exp}(D_{1}\cap N_{1})\subset M_{1}$. All
other required inclusions $\mathrm{Exp}(D_{1}\cap N_{j})\subset M_{1},\quad j=2,3,5$,
are proved similarly, and the inclusion $\mathrm{Exp}(D_{1})\subset M_{1}$
follows.
\item[(2)] Since $\widetilde{N}=D_{1}\cup D_{2}$ and $\widetilde{M}=M_{1}\cup M_{2}$,
the inclusion $\mathrm{Exp}(\widetilde{N})\subset\widetilde{M}$ follows
from item (1).
\item[(3)] We have $N^{\prime}=\left(N^{\prime}\cap N_{1}\right)\sqcup\left(N^{\prime}\cap N_{2}\right)\sqcup\left(N^{\prime}\cap N_{3}\right)\sqcup N_{4}$.
\\
Let $(\lambda,t)\in N^{\prime}\cap N_{1}=\left\{ (\lambda,t)\in\widehat{N}_{1}\quad\mid\quad t=\mathbf{t}(\lambda)\textrm{ or }\sin\frac{\gamma_{t/2}}{2}=0\right\} ,$
then similarly to the proof of item (1) we get $p=2K(k)$ or $\mathrm{sn}\tau=0$,
thus $z_{t}=0$ by (\ref{eq:3.10}). From (3.6) \cite{Max_Conj_SH2}
we get $R_{2}(q_{t})=\frac{2s_{1}}{1-k^{2}}\mathrm{dn}\tau\, f_{2}(p)\neq0$,
and therefore $x^{2}+y^{2}\neq0$. We proved that $\mathrm{Exp}(N^{\prime}\cap N_{1})\subset M^{\prime}$.
It follows similarly that $\mathrm{Exp}(N^{\prime}\cap N_{j})\subset M^{\prime},\quad j=2,3$.
Finally, if $(\lambda,t)\in\widehat{N}_{4},$ then 
\[
q_{t}=(x_{t},y_{t},z_{t})=(t,0,0)\in M^{\prime}.
\]
Consequently, $\mathrm{Exp}(N^{\prime})\subset M^{\prime}$.\hfill$\square$\end{enumerate}
\begin{theorem}
\label{thm:Max_Conj}For $\lambda\in\cup_{i=1}^{5}C_{i}$, we have
$t_{1}^{\mathrm{conj}}(\lambda)\geq t_{1}^{\mathrm{Max}}(\lambda)$. \end{theorem}
\begin{proof}
Apply equations (\ref{eq:2.20})--(\ref{eq:2.24}) and (\ref{eq:2.25})--(\ref{eq:2.28}).\hfill$\square$\end{proof}

\begin{proposition}
\label{prop:Exp Nondeg}The restriction $\mathrm{Exp}:\widetilde{N}\rightarrow\widetilde{M}$
is non-degenerate.\end{proposition}
\begin{proof}
From Theorem \ref{thm:Max_Conj}, $t_{1}^{\mathrm{conj}}(\lambda)\geq t_{1}^{\mathrm{Max}}(\lambda)$.
Since for any $\nu=(\lambda,t)\in\widetilde{N}$ we have $t<\mathbf{t}(\lambda)$
and therefore exponential mapping is non-degenerate $\forall\nu=(\lambda,t)\in\widetilde{N}$
. \hfill$\square$
\end{proof}

Hence we proved properties \textbf{P1, P2 }and \textbf{P3} for the
exponential mapping $\mathrm{Exp}:D_{i}\rightarrow M_{i}$. It only
remains to prove condition \textbf{P4} now to establish that the exponential
mapping $\mathrm{Exp}:D_{i}\rightarrow M_{i}$ is indeed a diffeomorphism.

\subsection{Diffeomorphic Properties of the Exponential Mapping}

In this subsection we prove that the exponential mapping $\mathrm{Exp}:D_{i}\rightarrow M_{i},\quad i=1,2$,
is proper. First we recall an equivalent formulation of the properness
property.
\begin{definition}
Let $X$ be a topological space and $\{x_{n}\}\subset X$ a sequence.
We write $x_{n}\rightarrow\partial X$ if there is no compact $K\subset X$
such that $x_{n}\in K$ for any $n\in\mathbb{N}$. \end{definition}
\begin{remark}
Let $X,Y$ be topological spaces and $F:X\rightarrow Y$ a continuous
mapping. The mapping $F$ is proper iff for any sequence $\{x_{n}\}\subset X$
there holds the implication:
\[
x_{n}\rightarrow\partial X\implies F(x_{n})\rightarrow\partial Y.
\]

\end{remark}
Below we apply this properness test to the mapping $\mathrm{Exp}:D_{1}\rightarrow M_{1}$. 
\begin{lemma}
\label{lem:dM1}Let $\{q_{n}\}\subset M_{1}$. We have $q_{n}\rightarrow\partial M_{1}$
iff there is a subsequence $\{n_{k}\}$ on which one of the conditions
holds:\end{lemma}
\begin{enumerate}
\item[(1)] $z\rightarrow0,$
\item[(2)] $z\rightarrow+\infty,$
\item[(3)] $x\rightarrow\infty,$
\item[(4)] $y\rightarrow\infty.$\end{enumerate}
\begin{proof}
Any compact set in $M_{1}$ is contained in a compact set $\left\{ q\in M_{1}\quad\mid\quad\varepsilon\leq z\leq\frac{1}{\varepsilon},\quad\left|x\right|\leq\frac{1}{\varepsilon},\quad\left|y\right|\leq\frac{1}{\varepsilon}\right\} $
for some $\varepsilon\in(0,1)$. \hfill$\square$\end{proof}

\begin{lemma}
\label{lem:dD1}Let $\{\nu_{n}\}\subset D_{1}$, then $\nu_{n}\rightarrow\partial D_{1}$
iff there is a subsequence $\{n_{k}\}$ on which one of the following
conditions hold:\end{lemma}
\begin{enumerate}
\item[(1)] $\gamma_{t/2}\rightarrow0,$
\item[(2)] $\gamma_{t/2}\rightarrow2\pi,$
\item[(3)] $c_{t/2}\rightarrow\infty,$
\item[(4)] $t\rightarrow0,$
\item[(5)] $t\rightarrow+\infty,$
\item[(6)] $\mathbf{t}(\lambda)-t\rightarrow0.$\end{enumerate}
\begin{proof}
Any compact set in $D_{1}$ is contained in a compact set 
\[
\left\{ \nu\in N\,\mid\,\gamma_{t/2}\in\left[\varepsilon,2\pi-\varepsilon\right],\left|c_{t/2}\right|\leq\frac{1}{\varepsilon},\, t\in[\varepsilon,\frac{1}{\varepsilon}],\,\mathbf{t}(\lambda)-t\geq\varepsilon\right\} ,
\]
for some $\varepsilon\in(0,1)$.\hfill$\square$ \end{proof}

\begin{proposition}
\label{prop:Proper}The mapping $\mathrm{Exp}:D_{i}\rightarrow M_{i},\quad i=1,2$,
is proper.\end{proposition}
\begin{proof}
In view of the reflections $\varepsilon^{j}$, it suffices to consider
the case $\mathrm{Exp}:D_{1}\rightarrow M_{1}$. Let $\left\{ \nu_{n}\right\} \subset D_{1}$,
$\nu_{n}\rightarrow\partial D_{1}$, we have to show that $q_{n}=\mathrm{Exp}(\nu_{n})\rightarrow\partial M_{1}$.
Taking into account decomposition (\ref{eq:D1decomp}), we can consider
the cases $\left\{ \nu_{n}\right\} \subset D_{1}\cap N_{j},\quad j=1,2,3,5$.

Let $\left\{ \nu_{n}\right\} \subset D_{1}\cap N_{1}$, $\nu_{n}\rightarrow\partial D_{1}$.
We will need the following formulas for the extremals $\lambda_{t}=e^{t\overrightarrow{H}}(\lambda),\quad\lambda\in C_{1}$,
obtained in \cite{Extremal_Pseudo_Euclid} and \cite{Max_Conj_SH2}:
\begin{eqnarray*}
\sin\frac{\gamma_{t}}{2} & = & s_{1}k\,\mathrm{sn}(\varphi_{t}),\\
\frac{c_{t}}{2} & = & k\,\mathrm{cn}(\varphi_{t}),\\
\sinh z_{t} & = & s_{1}\frac{k\,\mathrm{sn}p\,\mathrm{sn}\tau}{\Delta},\quad\Delta=1-k^{2}\mathrm{sn}^{2}p\,\mathrm{sn}^{2}\tau,\\
R_{2}(q_{t}) & = & f_{2}(p)\frac{2s_{1}}{1-k^{2}}\mathrm{dn}\tau,\quad f_{2}(p)=\mathrm{dn}p\,\mathrm{E}(p)-k^{2}\mathrm{sn}p\,\mathrm{cn}p.
\end{eqnarray*}
Notice that $p=\frac{t}{2},\quad\tau=\varphi+\frac{t}{2}$, and consider
all the cases (1)--(6) of Lemma \ref{lem:dD1}.\end{proof}

\begin{enumerate}
\item[(1)] If $\gamma_{t/2}\rightarrow0$, then $\sin\frac{\gamma_{t/2}}{2}=s_{1}k\,\mathrm{sn}\tau\rightarrow0$,
thus $\sinh z_{t}\rightarrow0$, so $z_{t}\rightarrow0$, hence $q_{n}\rightarrow\partial M_{1}$(Lemma
\ref{lem:dM1}, (1)).
\item[(2)] If $\gamma_{t/2}\rightarrow2\pi$, then $\sin\frac{\gamma_{t/2}}{2}=s_{1}k\,\mathrm{sn}\tau\rightarrow0$,
thus $\sinh z_{t}\rightarrow0$, so $z_{t}\rightarrow0$, hence $q_{n}\rightarrow\partial M_{1}$. 
\item[(3)] The case $c_{t/2}\rightarrow\infty$ is impossible.
\item[(4)] If $t\rightarrow0$, then $p\rightarrow0$, thus $z_{t}\rightarrow0$.
\item[(5)] Let $t\rightarrow+\infty$, then $p\rightarrow+\infty$. Since $p\in(0,2K(k))$
then $k\rightarrow1$. Denote $u=\mathrm{am}(p)\in(0,\pi)$. On a
subsequence we have $u\rightarrow\bar{u}\in[0,\pi]$ and we will suppose
so in the sequel.

\begin{enumerate}
\item If $\bar{u}\in[0,\pi)$, then $p=F(u,k)\rightarrow F(\bar{u},1)=\intop_{0}^{\bar{u}}\frac{dt}{\cos(t)}<+\infty$,
a contradiction. 
\item Let $\bar{u}=\frac{\pi}{2}$, thus $\mathrm{sn}p=\sin u\rightarrow1$,
$\mathrm{cn}p=\cos(u)\rightarrow c$.

\begin{enumerate}
\item If $\mathrm{sn}\tau\rightarrow1$, then $\Delta\rightarrow0$, thus
$z_{t}\rightarrow\infty$. 
\item Let $\mathrm{sn}\tau\rightarrow\bar{s}\neq1$, then $\mathrm{dn}\tau\rightarrow\sqrt{1-\bar{s}^{2}}\neq0$.
Denote 
\begin{eqnarray*}
g_{2}(u) & = & f_{2}(F(u,k))=\sqrt{1-k^{2}\sin^{2}u}\mathrm{E}(u,k)-k^{2}\sin(u)\cos(u).
\end{eqnarray*}
We prove now that $\frac{g_{2}(u)}{1-k^{2}}\rightarrow+\infty$, then
$\frac{f_{2}(u)}{1-k^{2}}\rightarrow+\infty$, thus $R_{2}(q_{t})\rightarrow\infty$,
so $x_{t}^{2}+y_{t}^{2}+z_{t}^{2}\rightarrow\infty$, whence $q_{t}\rightarrow\partial M_{1}$.
Denote $k^{\prime}=\sqrt{1-k^{2}}\rightarrow0$. We can suppose that
on a subsequence $\frac{\cos u}{k^{\prime}}\rightarrow\alpha\in[0,+\infty]$.
We have 
\begin{eqnarray*}
k^{2}\sin(u)\cos(u) & = & \sin(u)\cos(u)+o(k^{\prime2}),\\
\sqrt{1-k^{2}\sin^{2}u} & = & \sqrt{\cos^{2}u+k^{\prime2}-k^{\prime2}\cos^{2}u}.
\end{eqnarray*}
Now we estimate $E(u,k)$ from below:
\begin{eqnarray*}
E(u,k)-\sin(u) & = & \intop_{0}^{u}\sqrt{1-k^{2}\sin^{2}t}dt-\intop_{0}^{u}\cos(t)dt=\intop_{0}^{u}\frac{1-k^{2}\sin^{2}t-\cos^{2}t}{\sqrt{1-k^{2}\sin^{2}t}+\cos t}dt\\
 & > & \frac{1-k^{2}}{2}\intop_{0}^{u}\sin^{2}t\, dt\\
 & = & \frac{1-k^{2}}{4}\left(u-\frac{\sin(2u)}{2}\right)\\
 & = & \frac{\pi}{8}k^{\prime2}(1+o(1)).
\end{eqnarray*}

Thus,
\[
E(u,k)>\sin(u)+\frac{\pi}{8}k^{\prime2}(1+o(1)).
\]

\begin{enumerate}
\item Let $\alpha\in[0,+\infty).$ Then $\cos(u)=\alpha k^{\prime}+o(k^{\prime}),\quad\sin(u)=1+o(1),$
thus
\begin{eqnarray*}
k^{2}\sin(u)\cos(u) & = & \alpha k^{\prime}+o(k^{\prime}),\\
\sqrt{1-k^{2}\sin^{2}(u)} & = & \sqrt{1+\alpha^{2}}k^{\prime}+o(k^{\prime}),\\
E(u,k) & = & 1+o(1),\\
\sqrt{1-k^{2}\sin^{2}u}\, E(u,k) & = & \sqrt{1+\alpha^{2}}k^{\prime}+o(k^{\prime}),\\
g_{2}(u) & = & \left(\sqrt{1+\alpha^{2}}-\alpha\right)k^{\prime}+o(k^{\prime}),\\
\frac{g_{2}(u)}{k^{\prime2}} & = & \frac{\left(\sqrt{1+\alpha^{2}}-\alpha\right)}{k^{\prime}}(1+o(1))\rightarrow\infty,
\end{eqnarray*}
and the claim follows. 
\item Let $\alpha=+\infty$, thus $k^{\prime}=o(\cos(u))$. Then 
\begin{eqnarray*}
k^{2}\sin(u)\cos(u) & = & \sin(u)\cos(u)-k^{\prime2}\cos(u)+o\left(k^{\prime2}\cos(u)\right),\\
\sqrt{1-k^{2}\sin^{2}u} & = & \cos(u)\sqrt{1+\frac{k^{\prime2}}{\cos^{2}u}+o\left(\frac{k^{\prime2}}{\cos^{2}u}\right)}\\
 & = & \cos(u)+\frac{1}{2}\frac{k^{\prime2}}{\cos(u)}+o\left(\frac{k^{\prime2}}{\cos(u)}\right),\\
\sqrt{1-k^{2}\sin^{2}u}\, E(u,k) & > & \cos(u)\sin(u)+\frac{1}{2}\frac{k^{\prime2}}{\cos(u)}+o\left(\frac{k^{\prime2}}{\cos(u)}\right),\\
g_{2}(u) & > & \frac{1}{2}\frac{k^{\prime2}}{C}(1+o\left(1\right)),\\
\frac{g_{2}(u)}{k^{\prime2}} & > & \frac{1}{2C}(1+o\left(1\right))\rightarrow+\infty,
\end{eqnarray*}
and the claim follows.
\end{enumerate}
\item Let $u\in(0,\pi)$, then $f_{2}(p)=g_{2}(u)\rightarrow\left|\cos\bar{u}\right|\left(E(\bar{u},1)+\sin\bar{u}\right)>0,$
thus 
\[
\frac{f_{2}(p)}{\sqrt{1-k^{2}}}\rightarrow+\infty.
\]
Since $\frac{\mathrm{dn}\tau}{\sqrt{1-k^{2}}}\geq1$, then $R_{2}(q_{t})\rightarrow\infty$,
so $x_{t}^{2}+y_{t}^{2}+z_{t}^{2}\rightarrow\infty$, whence $q_{t}\rightarrow\partial M_{1}$.
\item If $\bar{u}=\pi$, then $\mathrm{sn}p=\sin(u)\rightarrow0$, thus
$z_{t}\rightarrow0$.
\end{enumerate}
\end{enumerate}

\item[(6)] Let $\mathbf{t}(\lambda)-t\rightarrow0$. Recall that $\mathbf{t}(\lambda)=4K(k)$
for $\lambda\in C_{1}$, thus $4K(k)-t\rightarrow0$. Since $k\in(0,1)$,
then there is a subsequence $\{n_{m}\}$ on which $k\rightarrow\bar{k}\in[0,1]$.
If $\bar{k}\in[0,1)$, then $K(k)\rightarrow K(\bar{k})<+\infty$, thus $t\rightarrow4K(\bar{k})$, so $p=2K(\bar{k})$. Consequently, $\sinh z_{t}\rightarrow0$, whence $q_{n}\rightarrow\partial M_{1}$ (Lemma \ref{lem:dM1}, (1)). If $\bar{k}=1$, then $K(k)\rightarrow+\infty$,
thus $t\rightarrow+\infty$, $q_{n}\rightarrow\partial M_{1}$ by item (5).
\end{enumerate}

Consequently, in each of the cases (1)--(6) of Lemma \ref{lem:dD1}
we get $q_{n}\rightarrow\partial M_{1}$ for a sequence $\{\nu_{n}\}\subset D_{1}\cap N_{1},\quad\nu_{n}\rightarrow\partial D_{1}$.
All the rest cases $\{\nu_{n}\}\subset D_{1}\cap N_{j},\quad j=2,3,5,$
are considered similarly. 

Summing up, for any sequence $\{\nu_{n}\}\subset D_{1}$ with $\nu_{n}\rightarrow\partial D_{1}$
we have $\mathrm{Exp}(\nu_{n})\rightarrow\partial M_{1}$. Thus the
mapping $\mathrm{Exp}:D_{1}\rightarrow M_{1}$ is proper. \hfill$\square$ 

Now we get the main result of this section.
\begin{theorem}
\label{thm:Exp Di Diffeo}The mapping $\mathrm{Exp}:D_{i}\rightarrow M_{i},\quad i=1,2$,
is a diffeomorphism.\end{theorem}
\begin{proof}
All of the conditions \textbf{P1--P4 }are satisfied for the mapping
$\mathrm{Exp}:D_{1}\rightarrow M_{1}$:\end{proof}
\begin{itemize}
\item $D_{1}\subset N$ and $M_{1}\subset M$ are open subsets thus 3-dimensional
manifolds (Proposition \ref{prop:Di-topol}, Proposition \ref{prop:Mi-topol}),
\item \textbf{P1 - }$D_{1}$ is connected (Proposition \ref{prop:Di-topol}),
\item \textbf{P2 - $M_{1}$ }is connected and simply connected\textbf{ }(Proposition
\ref{prop:Mi-topol}),
\item \textbf{P3 - }$\left.\mathrm{Exp}\right|_{D_{1}}$ is non-degenerate
(Proposition \ref{prop:Exp Nondeg}),
\item \textbf{P4 - }$\mathrm{Exp}:D_{1}\rightarrow M_{1}$ is proper (Proposition
\ref{prop:Proper}).
\end{itemize}
Thus $\mathrm{Exp}:D_{1}\rightarrow M_{1}$ is a diffeomorphism. By
virtue of the reflections, $\mathrm{Exp}:D_{2}\rightarrow M_{2}$
is a diffeomorphism as well.

\hfill$\square$ 
\begin{corollary}
\label{cor:Diffeo}The exponential mapping $\mathrm{Exp}:\widetilde{N}\rightarrow\widetilde{M}$
is a diffeomorphism.\end{corollary}
\begin{proof}
Follows from Theorem \ref{thm:Exp Di Diffeo}.\hfill$\square$ 
\end{proof}

\subsection{Cut Time}

Now we can prove that inequality (\ref{eq:tCutbound}) is in fact
an equality for $\lambda\in C\backslash C_{4}$. 
\begin{theorem}
\label{thm:Cut_time}If $\lambda\in C\backslash C_{4}$, then $t_{\mathrm{cut}}(\lambda)=\mathbf{t}(\lambda)$. \end{theorem}
\begin{proof}
Let $\lambda\in C\backslash C_{4}=\cup_{i=1}^{3}C_{i}\cup C_{5}$.
 In view of inequality (\ref{eq:tCutbound}), it remains to prove
that $t_{\mathrm{cut}}(\lambda)\geq\mathbf{t}(\lambda)$. Take any
$t_{1}\in(0,\mathbf{t}(\lambda)).$We need to prove that the geodesic
$\mathrm{Exp}(\lambda,t)$ is optimal on the segment $t\in[0,t_{1}].$ 

Consider first the case $\lambda\in\cup_{i=1}^{3}C_{i}$. If $\sin\frac{\gamma_{t_{1}/2}}{2}\neq0$,
then $(\lambda,t_{1})\in\widetilde{N}$, and $q_{1}=\mathrm{Exp}(\lambda,t_{1})\in\widetilde{M}$.
By virtue of Proposition \ref{prop:Exp Di Mi} and Theorem \ref{thm:Exp Di Diffeo},
the point $q_{1}$ has a unique preimage under the mapping $\mathrm{Exp}:\widehat{N}\rightarrow\widehat{M}$.
Thus the geodesic $\mathrm{Exp}(\lambda,t)$ is optimal on the segment
$t\in[0,t_{1}]$. 

If $\lambda\in\cup_{i=1}^{3}C_{i}$ and $\sin\frac{\gamma_{t_{1}/2}}{2}=0$,
then we can choose $t_{2}\in(t_{1},\mathbf{t}(\lambda))$ such that
$\sin\frac{\gamma_{t_{2}/2}}{2}\neq0$. By the argument of the preceding
paragraph, the geodesic $\mathrm{Exp}(\lambda,t)$ is optimal at the
segment $[0,t_{2}]$, thus at the segment $[0,t_{1}]\subset[0,t_{2}]$
as well. 

Finally, if $\lambda\in C_{5}$, then $(\lambda,t_{1})\in\widetilde{N}$,
and the geodesic $\mathrm{Exp}(\lambda,t),\quad t\in[0,t_{1}]$, is
optimal as above. 

We proved that $t_{\mathrm{cut}}(\lambda)\geq\mathbf{t}(\lambda)$,
thus $t_{\mathrm{cut}}(\lambda)=\mathbf{t}(\lambda)$ for any $\lambda\in C\backslash C_{4}$.\hfill$\square$
\end{proof}

We will be able to prove the equality $t_{\mathrm{cut}}(\lambda)=\mathbf{t}(\lambda)$
for $\lambda\in C_{4}$ below after the description of the structure
of the exponential mapping $\mathrm{Exp}:N^{\prime}\rightarrow M^{\prime}$.
The geodesic $\mathrm{Exp}(\lambda,t),\quad\lambda\in C_{4}$, requires
a separate study since it belongs to the set $M^{\prime}$ for all
$t>0$.

Intuitively, Theorem \ref{thm:Cut_time} establishes the fact that
since $\mathrm{Exp}:\widetilde{N}\rightarrow\widetilde{M}$ is a diffeomorphism,
hence upto time $t<\mathbf{t}(\lambda)$ there is a unique point $\nu=(\lambda,s)\in\widetilde{N}$
that is mapped to a unique extremal trajectory $q_{s}=\mathrm{Exp}(\lambda,s)\in\widetilde{M}$
that joins $q_{0}\in M$ to $q_{1}\in\widetilde{M}\subset M$. Hence,
the trajectory $q_{s}=\mathrm{Exp}(\lambda,s)\in\widetilde{M}$ is
optimal and therefore $t_{\mathrm{cut}}(\lambda)=\mathbf{t}(\lambda)$.
It therefore follows that optimal synthesis in the domain $\widetilde{M}$
is given by:
\[
u_{i}(q)=h_{i}(\lambda),\quad i=1,2,\quad(\lambda,t)=\mathrm{Exp}^{-1}(q)\in\widetilde{N},\quad q\in\widetilde{M},
\]
 where $u_{i}$ are the control variables (i.e., translational and
rotational velocities) and $h_{i}$ are the optimal controls defined
in (4.8) \cite{Extremal_Pseudo_Euclid}.

\section{Exponential Mapping on the Boundary of Diffeomorphic Domains}

Until now we have studied the mapping $\mathrm{Exp}:\widetilde{N}\rightarrow\widetilde{M}$
and proved that it is a diffeomorphism. This allowed us to prove that
the cut time $t_{\mathrm{cut}}(\lambda)=t_{1}^{\mathrm{Max}}(\lambda),\quad\lambda\in C\backslash C_{4}$.
In this section we obtain the global structure of the exponential
mapping in order to characterize the cut locus and the Maxwell strata
and to construct the optimal synthesis. Specifically we study the
mapping $\mathrm{Exp}:N^{\prime}\rightarrow M^{\prime}$ where:
\begin{eqnarray*}
N^{\prime} & = & \left\{ (\lambda,t)\in\cup_{i=1}^{3}N_{i}\quad\mid\quad t=t_{1}^{\mathrm{Max}}(\lambda)\quad\textrm{or}\quad\sin\left(\frac{\gamma_{t/2}}{2}\right)=0\right\} \cup\left\{ (\lambda,t)\in N_{4}\quad\mid\quad t\leq2\pi=t_{1}^{\mathrm{conj}}(\lambda)\right\} ,\\
M^{\prime} & = & \left\{ q\in M\quad\mid\quad x^{2}+y^{2}\neq0,\quad z=0\right\} .
\end{eqnarray*}

\subsection{Stratification of $N^{\prime}$}

We define subsets $N_{j}^{\prime}\subset N^{\prime},\quad j=1,\ldots,40$,
as follows:
\begin{itemize}
\item for $j\in\left\{ 1,9,17,21,25,29\right\} $ the sets $N_{j}^{\prime}$
are given by Table \ref{tab:Nj}, for $j=35$ by Table \ref{tab:N35}
and for $j\in\left\{ 33,39\right\} $ by Table \ref{tab:N33_39},
\item for all the rest $j$ the set $N_{j}^{\prime}$ are defined by the
action of reflections $\varepsilon^{i}$ as in (\ref{eq:epsiNj})--(\ref{eq:eps4Nj}):
\end{itemize}
\begin{table}
\begin{centering}
\begin{tabular}{|c|c|c|c|c|}
\hline 
$j$ & $\lambda$ & $p$ & $\tau$ & $k$\tabularnewline
\hline 
\hline 
1 & $C_{1}^{0}$ & $2K$ & $(0,K)$ & $(0,1)$\tabularnewline
\hline 
9 & $C_{2}^{+}$ & $2K$ & $(0,K)$ & $(0,1)$\tabularnewline
\hline 
17 & $C_{1}^{0}$ & $2K$ & $K$ & $(0,1)$\tabularnewline
\hline 
21 & $C_{1}^{0}$ & $2K$ & $0$ & $(0,1)$\tabularnewline
\hline 
25 & $C_{2}^{+}$ & $2K$ & $0$ & $(0,1)$\tabularnewline
\hline 
29 & $C_{2}^{+}$ & $2K$ & $K$ & $(0,1)$\tabularnewline
\hline 
\end{tabular}\protect\caption{\label{tab:Nj}Decomposition $N_{j}^{\prime},\quad j\in\left\{ 1,9,17,21,25,29\right\} $}

\par\end{centering}

\end{table}
 
\begin{table}
\centering{}%
\begin{tabular}{|c|c|c|c|}
\hline 
$\lambda$ & $p$ & $\tau$ & $k$\tabularnewline
\hline 
\hline 
$C_{1}^{0}$ & $(0,2K)$ & 0 & $\left(0,1\right)$\tabularnewline
\hline 
$C_{2}^{+}$ & $(0,2K)$ & 0 & $\left(0,1\right)$\tabularnewline
\hline 
$C_{3}^{0+}$ & $(0,+\infty)$ & 0 & 1\tabularnewline
\hline 
\end{tabular}\protect\caption{\label{tab:N35}Decomposition $N_{j}^{\prime},\quad j=35$}
\end{table}
\begin{table}
\centering{}%
\begin{tabular}{|c|c|c|}
\hline 
$j$ & $\lambda$ & $t$\tabularnewline
\hline 
\hline 
33 & $C_{4}^{0}$ & $2\pi$\tabularnewline
\hline 
39 & $C_{4}^{0}$ & $(0,2\pi)$\tabularnewline
\hline 
\end{tabular}\protect\caption{\label{tab:N33_39}Decomposition $N_{j}^{\prime},\quad j\in\left\{ 33,39\right\} $}
\end{table}
\begin{eqnarray}
\varepsilon^{i}\left(N_{j}^{\prime}\right) & = & N_{j+i}^{\prime},\quad i=1,\ldots,7,\quad j=1,9,\label{eq:epsiNj}\\
\varepsilon^{2i}\left(N_{17}^{\prime}\right) & = & N_{17+i}^{\prime},\quad i=1,2,3,\label{eq:eps2iNj}\\
\varepsilon^{2+i}\left(N_{j}^{\prime}\right) & = & N_{j+i}^{\prime},\quad i=1,2,3,\quad j=21,25,29,35,\label{eq:eps2+iNj}\\
\varepsilon^{4}\left(N_{j}^{\prime}\right) & = & N_{j+1}^{\prime},\quad j=33,39.\label{eq:eps4Nj}
\end{eqnarray}
The following stratification of the set $N^{\prime}$ follows from
the definition of the sets $N_{j}^{\prime}$.
\begin{lemma}
The stratification of $N^{\prime}$ shown in Figures \text{\emph{\ref{fig:N'decomp1}}},\text{\emph{\ref{fig:N'decomp2}}}
is given as:
\begin{equation}
N^{\prime}=\sqcup_{j=1}^{40}N_{j}^{\prime}.\label{eq:N'Decomposn}
\end{equation}
\begin{figure}
\begin{centering}
\includegraphics[scale=0.6]{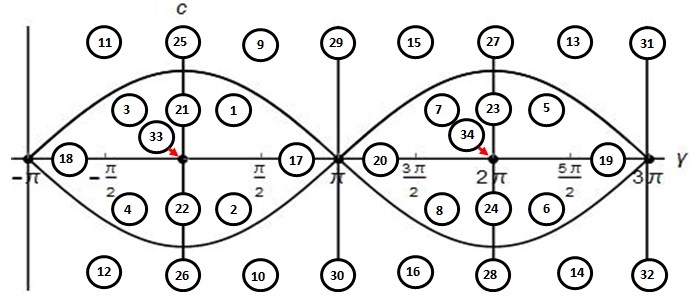}\protect\caption{\label{fig:N'decomp1}The sets $N_{j}^{\prime}$ with $t=t_{1}^{\mathrm{Max}}(\lambda)\quad\textrm{or}\quad\sin\left(\frac{\gamma_{t/2}}{2}\right)=0$}

\par\end{centering}

\end{figure}
\begin{figure}
\centering{}\includegraphics[scale=0.6]{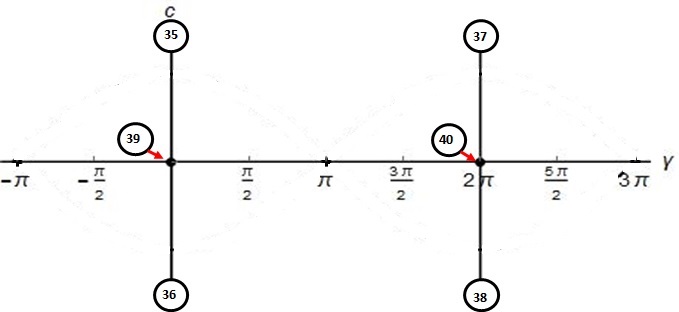}\protect\caption{\label{fig:N'decomp2}The sets $N_{j}^{\prime}$ with $t<t_{1}^{\mathrm{Max}}(\lambda)$,
$\sin\frac{\gamma_{t/2}}{2}=0$}
\end{figure}

\end{lemma}
From Figures \ref{fig:N'decomp1}, \ref{fig:N'decomp2} we see the
sets $N_{j}^{\prime}$ given in Tables \ref{tab:Nj}, \ref{tab:N35},
\ref{tab:N33_39} pertain to the quadrant of the phase portrait of
vertical subsystem for which $\lambda=(\gamma,c)\in C$ such that
$\gamma\in[0,\pi]$ and $c\in[0,\infty)$. For $\lambda=(\gamma,c)$
in other parts of phase portrait, the sets $N_{j}^{\prime}$ are obtained
by the reflection symmetries (\ref{eq:epsiNj})--(\ref{eq:eps4Nj})
of the vertical subsystem.

\subsection{Stratification of a Quadrant of the Plane $z=0$}

Define the following curves and points in the quadrant $Q=\left\{ (x,y)\in\mathbb{R}^{2}\quad\mid\quad x\geq0,\, y\leq0\right\} $
(see Figure \ref{fig:g1g5}): 
\begin{figure}
\centering{}\includegraphics[scale=0.5]{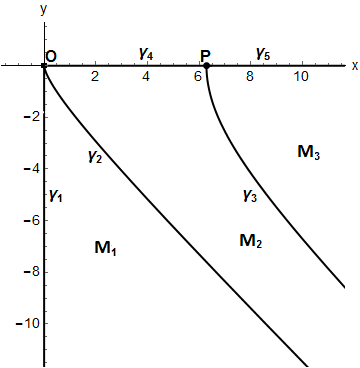}\protect\caption{\label{fig:g1g5}Stratification of the quadrant $Q$}
\end{figure}
\begin{flalign*}
\gamma_{1}:\quad & x=0,\quad y=y_{1}(k)=-\frac{4a(k)}{\sqrt{1-k^{2}}},\quad k\in(0,1),\\
\gamma_{2}:\quad & x=x_{2}(k)=\frac{4k\, a(k)}{1-k^{2}},\quad y=y_{2}(k)=-\frac{4a(k)}{1-k^{2}},\quad k\in(0,1),\\
\gamma_{3}:\quad & x=x_{3}(k)=\frac{4}{1-k^{2}}E(k),\quad y=y_{3}(k)=-\frac{4k}{1-k^{2}}E(k),\quad k\in(0,1),\\
\gamma_{4}:\quad & x=x_{4}(t)=t,\quad y=0,\quad t\in(0,2\pi),\\
\gamma_{5}:\quad & x=x_{5}(k)=\frac{4}{\sqrt{1-k^{2}}}E(k),\quad y=0,\quad k\in(0,1),\\
P:\quad & x=2\pi,\quad y=0,\\
O:\quad & x=0,\quad y=0, & {}
\end{flalign*}
where $a(k)=E(k)-(1-k^{2})K(k),\quad k\in(0,1)$. The curves $\gamma_{1},\ldots,\gamma_{5}$
result from substitution of $t=t_{1}^{\mathrm{Max}}(\lambda)$, and
$\varphi=\tau-p$ from Table \ref{tab:Nj} in the equations of extremal
trajectories for $\lambda\in\cup_{i=1}^{5}C_{i}$. The curves $\gamma_{1},\ldots,\gamma_{5}$
and the point $P$ are the images of certain sets $\mathrm{Exp}\left(N_{j}^{\prime}\right)$
under the projection
\begin{equation}
p:\left\{ q\in M\quad\mid\quad z=0\right\} \rightarrow\mathbb{R}_{x,y}^{2},\quad\left(x,y,0\right)\mapsto(x,y).\label{eq:P}
\end{equation}
\begin{flalign*}
\gamma_{1}= & p\circ\mathrm{Exp}\left(N_{29}^{\prime}\right),\\
\gamma_{2}= & p\circ\mathrm{Exp}\left(N_{25}^{\prime}\right),\\
\gamma_{3}= & p\circ\mathrm{Exp}\left(N_{21}^{\prime}\right),\\
\gamma_{4}= & p\circ\mathrm{Exp}\left(N_{39}^{\prime}\right),\\
\gamma_{5}= & p\circ\mathrm{Exp}\left(N_{17}^{\prime}\right),\\
P= & p\circ\mathrm{Exp}\left(N_{33}^{\prime}\right).
\end{flalign*}
These equalities can be verified easily. From \cite{Max_Conj_SH2}
we know that the first Maxwell points with $t=t_{1}^{\mathrm{Max}}(\lambda)$
and conjugate points with $t=t_{1}^{\mathrm{Max}}(\lambda)$ and $\mathrm{sn}\tau\,\mathrm{cn}\tau=0$
lie in the plane $z=0$. Hence, the curves $\gamma_{1},\ldots,\gamma_{5}$
decompose the fourth quadrant of the plane $z=0$ into various regions
(see Figure \ref{fig:g1g5}). The regularity and mutual disposition
of the curves $\gamma_{1},\ldots,\gamma_{5}$ are described in the
following lemmas. 
\begin{lemma}
\label{lem:a(k)}The function $a(k)$ satisfies the following properties:
\begin{eqnarray}
a:(0,1) & \rightarrow & (0,1)\textrm{ is a diffeomorphism,}\label{eq:adiff}\\
k & \rightarrow & 0\implies a(k)=\frac{\pi}{4}k^{2}+o(k^{2}),\label{eq:ak0}\\
k & \rightarrow & 1-0\implies a(k)=1-\frac{1}{2}k^{\prime2}\ln\left(\frac{1}{k^{\prime}}\right)+O(k^{\prime2})\label{eq:ak1}
\end{eqnarray}
where $k^{\prime}=\sqrt{1-k^{2}}$. Moreover, the function $a(k)$
is convex.\end{lemma}
\begin{proof}
If $k\rightarrow0$, then 
\begin{eqnarray*}
K(k) & = & \frac{\pi}{2}\left(1+\frac{k^{2}}{4}\right)+o(k^{2}),\\
E(k) & = & \frac{\pi}{4}\left(1-\frac{k^{2}}{4}\right)+o(k^{2}),
\end{eqnarray*}
which gives asymptotics (\ref{eq:ak0}). If $k\rightarrow1-0$, then
\begin{eqnarray*}
K(k) & = & \ln\left(\frac{1}{k^{\prime}}\right)+o(k^{\prime}),\\
E(k) & = & 1+\frac{1}{2}k^{\prime2}\ln\left(\frac{1}{k^{\prime}}\right)+O(k^{\prime2}),
\end{eqnarray*}
which gives asymptotics (\ref{eq:ak1}). Finally, property (\ref{eq:adiff})
follows since
\begin{eqnarray*}
\frac{da}{dk} & = & k\, K(k)>0,\\
\lim_{k\rightarrow0}a(k) & = & 0,\\
\lim_{k\rightarrow1-0}a(k) & = & 1.
\end{eqnarray*}
The function $a(k)$ is convex since $\frac{da}{dk}=k\, K(k)$ increases
$\forall k\in(0,1)$. \hfill$\square$\end{proof}

\begin{lemma}
\label{lem:g1}The function $y=y_{1}(k)$ defines a diffeomorphism
$y_{1}:(0,1)\rightarrow(-\infty,0).$ Moreover,
\begin{eqnarray}
\lim_{k\rightarrow0^{+}}y_{1}(k) & = & 0,\label{eq:lim_y1_0}\\
\lim_{k\rightarrow1^{-}}y_{1}(k) & = & -\infty.\label{eq:lim_y1_inf}
\end{eqnarray}
\end{lemma}
\begin{proof}
The function $y=y_{1}(k)$ is a strictly decreasing function with:
\[
\frac{dy_{1}}{dk}=\frac{-4kE(k)}{(1-k^{2})^{\frac{3}{2}}}<0,\quad k\in(0,1).
\]
Further, Lemma \ref{lem:a(k)} yields the asymptotics:
\begin{eqnarray*}
k & \rightarrow & 0\implies y_{1}(k)=\frac{-4a(k)}{\sqrt{1-k^{2}}}\rightarrow0,\\
k & \rightarrow & 1-0\implies y_{1}(k)\thicksim-\frac{4}{k^{\prime}}\rightarrow-\infty,
\end{eqnarray*}
and the statement of this lemma follows. \hfill$\square$\end{proof}

\begin{lemma}
\label{lem:g4}The function $x=x_{4}(t)$ defines a diffeomorphism
$x_{4}:(0,2\pi)\rightarrow(0,2\pi)$. Moreover, 
\begin{eqnarray*}
\lim_{t\rightarrow0^{+}}x_{4}(t) & = & 0,\\
\lim_{k\rightarrow2\pi^{-}}x_{4}(t) & = & 2\pi.
\end{eqnarray*}
\end{lemma}
\begin{proof}
Clearly $x_{4}(t)$ is a smooth bijection with a smooth inverse. Hence
it is a diffeomorphsim. The limits can be calculated by direct substitution
in $x_{4}(t)$. \hfill$\square$\end{proof}

\begin{lemma}
\label{lem:g5}The function $x=x_{5}(k)$ defines a diffeomorphism
$x_{5}:(0,1)\rightarrow(2\pi,+\infty).$ Moreover,
\begin{eqnarray*}
\lim_{k\rightarrow0^{+}}x_{5}(k) & = & 2\pi,\\
\lim_{k\rightarrow1^{-}}x_{5}(k) & = & +\infty.
\end{eqnarray*}
\end{lemma}
\begin{proof}
The function $x=x_{5}(k)$ is a strictly decreasing function with:
\[
\frac{dx_{5}}{dk}=\frac{4a(k)}{k(1-k^{2})^{\frac{3}{2}}}>0,
\]
and 
\begin{eqnarray*}
k & \rightarrow & 0\implies E(k)\rightarrow\frac{\pi}{2}\implies x_{5}(k)\rightarrow2\pi,\\
k & \rightarrow & 1-0\implies E(k)\rightarrow1\implies x_{5}(k)\rightarrow+\infty,
\end{eqnarray*}
and the statement of the lemma follows. \hfill$\square$\end{proof}

\begin{lemma}
\label{lem:g2}The functions $x=x_{2}(k),\quad y=y_{2}(k)\quad k\in(0,1),$
define parametrically a function $x=x_{2}(y)$ which is a diffeomorphism
$x_{2}:(-\infty,0)\rightarrow(0,+\infty)$ with $\lim_{y\rightarrow-\infty}x_{2}(y)=+\infty,\quad\lim_{y\rightarrow0^{-}}x_{2}(y)=0$.
Moreover,
\begin{equation}
-y-2<x_{2}(y)<-y,\quad y\in(-\infty,0).\label{eq:x2bound}
\end{equation}
The curve $\gamma_{2}$ is convex, has near the origin the asymptotics
\begin{equation}
y=-\pi^{\frac{1}{3}}x^{\frac{2}{3}}+o\left(x^{\frac{2}{3}}\right),\quad x\rightarrow0,\label{eq:g2as}
\end{equation}
and has an asymptote $y+x+2=0$ as $x\rightarrow\infty$.\end{lemma}
\begin{proof}
Notice that
\begin{eqnarray*}
k & \rightarrow & 0\implies x_{2}(k)\rightarrow0,\quad y_{2}(k)\rightarrow0,\\
k & \rightarrow & 1\implies x_{2}(k)\rightarrow+\infty,\quad y_{2}(k)\rightarrow-\infty.
\end{eqnarray*}
Also,
\begin{eqnarray*}
\frac{dx_{2}}{dk} & = & \frac{4\left(\left(1+k^{2}\right)E(k)-(1-k^{2})K(k)\right)}{\left(1-k^{2}\right)^{2}}=\frac{4\left(a(k)+k^{2}E(k)\right)}{k\left(1-k^{2}\right)^{2}}>0,\\
\frac{dy_{2}}{dk} & = & -\frac{4k\left(2E(k)-(1-k^{2})K(k)\right)}{\left(1-k^{2}\right)^{2}}=-\frac{4k\left(a(k)+E(k)\right)}{\left(1-k^{2}\right)^{2}}<0,
\end{eqnarray*}
thus the functions $x_{2}(k)$ and $y_{2}(k)$ define diffeomorphisms
$x_{2}:(0,1)\rightarrow(0,+\infty)$ and $y_{2}:(0,1)\rightarrow(-\infty,0)$.
So these functions define parametrically the diffeomorphism 
\begin{eqnarray*}
x & = & x_{2}(y),\quad y\in(-\infty,0),\quad x\in(0,+\infty),\\
y & = & y_{2}(x),\quad x\in(0,+\infty),\quad y\in(-\infty,0).
\end{eqnarray*}
Notice that 
\begin{eqnarray*}
\lim_{y\rightarrow-\infty}x_{2}(y) & = & \lim_{k\rightarrow1}x_{2}(k)=+\infty,\\
\lim_{y\rightarrow0-}x_{2}(y) & = & \lim_{k\rightarrow0-}x_{2}(k)=0.
\end{eqnarray*}
Now we show that the curve $\gamma_{2}$ is convex. We have
\begin{eqnarray}
\frac{dy_{2}}{dx} & = & \frac{dy_{2}/dk}{dx_{2}/dk}=\alpha(k),\nonumber \\
\alpha(k) & = & -k\frac{2E(k)-(1-k^{2})K(k)}{(1+k^{2})E(k)-(1-k^{2})K(k)},\label{eq:alpha(k)}\\
\frac{d\alpha}{dk} & = & -\left(1-k^{2}\right)\frac{3E^{2}(k)-(5-k^{2})E(k)\, K(k)+2(1-k^{2})K^{2}(k)}{\left(\left(1+k^{2}\right)E(k)-\left(1-k^{2}\right)K(k)\right)^{2}}.\label{eq:da(k)}
\end{eqnarray}
Since $a(k)=E(k)-\left(1-k^{2}\right)K(k)\in(0,1)$, then $\frac{E(k)}{K(k)}\in\left(\left(1-k^{2}\right),1\right)$.
But the numerator of the function $t=\frac{E(k)}{K(k)}\mapsto3t^{2}-\left(5-k^{2}\right)t+2\left(1-k^{2}\right)$
is negative for $t\in\left(\left(1-k^{2}\right),1\right)$ thus the
numerator of fraction (\ref{eq:da(k)}) is positive. Therefore, $\frac{d\alpha}{dk}>0$,
i.e., $\frac{dy_{2}}{dx}$ is increasing for $k\in(0,1)$ and also
increasing for $x\in(0,+\infty)$. Thus the function $y_{2}(x)$ and
its graph, i.e., the curve $\gamma_{2}$, are convex. The second inequality
in (\ref{eq:x2bound}) follows since
\[
\frac{x_{2}(k)}{y_{2}(k)}=-k>-1,\quad k\in(0,1).
\]
The first inequality in (\ref{eq:x2bound}) and existence of the asymptote
$y+x+2=0$ follows from equalities:
\begin{eqnarray*}
\lim_{k\rightarrow1-}\frac{y_{2}(k)}{x_{2}(k)} & = & -1,\\
\lim_{k\rightarrow1-}\left(y_{2}(x)+x_{2}(y)\right) & = & -2,\\
\left(y_{2}(x)+x_{2}(y)\right)+2 & = & \frac{2}{1+k}\left(1+k-2a(k)\right)>0,
\end{eqnarray*}
since $a(k)<k<\frac{1+k}{2}$ for $k\in(0,1)$. Finally asymptotics
(\ref{eq:g2as}) follows since
\[
x_{2}(k)=\pi k^{3}+o(k^{3}),\quad y_{2}(k)=-\pi k^{2}+o(k^{2}),\quad k\rightarrow0.
\]
 \hfill$\square$

A plot of the curve $\gamma_{2}$ with its bounds given by (\ref{eq:x2bound})
is shown in Figure \ref{fig:g2}.
\begin{figure}
\centering{}\includegraphics[scale=0.6]{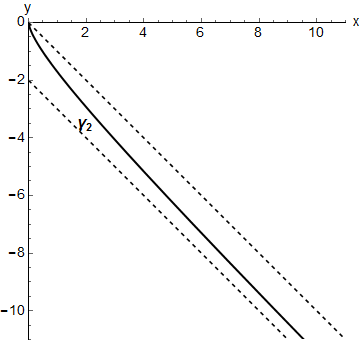}\protect\caption{\label{fig:g2}The curve $\gamma_{2}$ and its bounds $y+x=-2,\quad y+x=0$.}
\end{figure}
\end{proof}

\begin{lemma}
\label{lem:g3}The functions $x=x_{3}(k),\quad y=y_{3}(k),$ define
parametrically a function $x=x_{3}(y)$ which is a diffeomorphism
$x_{3}:(-\infty,0)\rightarrow(2\pi,+\infty)$ with $\lim_{y\rightarrow-\infty}x_{3}(y)=+\infty,\quad\lim_{y\rightarrow0^{+}}x_{3}(y)=2\pi$.
Moreover,
\begin{equation}
x_{3}(y)>2\pi,\quad x_{3}(y)>2-y,\quad y\in(-\infty,0).\label{eq:x3bound}
\end{equation}
The curve $\gamma_{3}$ is convex and has an asymptote $y+x=2$ as
$x\rightarrow\infty$. \end{lemma}
\begin{proof}
Notice that
\begin{eqnarray*}
k & \rightarrow & 0\implies x_{3}(k)\rightarrow2\pi,\quad y_{3}(k)\rightarrow0,\\
k & \rightarrow & 1\implies x_{3}(k)\rightarrow+\infty,\quad y_{3}(k)\rightarrow-\infty.
\end{eqnarray*}
Furthermore,
\begin{eqnarray*}
\frac{dx_{3}}{dk} & = & \frac{4\left(\left(1+k^{2}\right)E(k)-(1-k^{2})K(k)\right)}{k\left(1-k^{2}\right)^{2}}=\frac{4\left(a(k)+k^{2}E(k)\right)}{k\left(1-k^{2}\right)^{2}}>0,\\
\frac{dy_{3}}{dk} & = & -\frac{4\left(2E(k)-(1-k^{2})K(k)\right)}{k\left(1-k^{2}\right)^{2}}=-\frac{4\left(a(k)+E(k)\right)}{k\left(1-k^{2}\right)^{2}}<0,
\end{eqnarray*}
thus the functions $x_{3}(k)$ and $y_{3}(k)$ define diffeomorphisms
$x_{3}:(0,1)\rightarrow(2\pi,+\infty)$ and $y_{3}:(0,1)\rightarrow(-\infty,0)$.
So these functions define parametrically a diffeomorphism 
\begin{eqnarray*}
x & = & x_{3}(y),\quad y\in(-\infty,0),\quad x\in(2\pi,+\infty).
\end{eqnarray*}
Notice that 
\begin{eqnarray*}
\lim_{y\rightarrow-\infty}x_{3}(y) & = & \lim_{k\rightarrow1}x_{3}(k)=+\infty,\\
\lim_{y\rightarrow0+}x_{3}(y) & = & \lim_{k\rightarrow0+}x_{3}(k)=2\pi.
\end{eqnarray*}
Since $\frac{dx_{3}}{dk}>0$, therefore $x_{3}(k)>2\pi$ for $k\in(0,1)$,
which gives the first inequality in (\ref{eq:x3bound}). The second
inequality in (\ref{eq:x3bound}) and existence of the asymptote $y+x=2$
follow from the equalities:
\begin{eqnarray*}
\lim_{k\rightarrow1}\frac{y_{3}(k)}{x_{3}(k)} & = & -1,\\
\lim_{k\rightarrow1}\left(y_{3}(x)+x_{3}(y)\right) & = & 2,\\
\left(y_{3}(x)+x_{3}(y)\right)-2 & = & \frac{4}{1+k}\left(E(k)-\frac{1+k}{2}\right)>0.
\end{eqnarray*}
Finally, convexity of the curve $\gamma_{3}$ follows since
\[
\frac{dy_{3}}{dx}=\frac{dy_{3}/dk}{dx_{3}/dk}=\alpha(k),
\]
where $\alpha(k)$ is given by (\ref{eq:alpha(k)}), which is increasing
by the proof of Lemma \ref{lem:g2}.  \hfill$\square$

A plot of the curve $\gamma_{3}$ with its bounds given by (\ref{eq:x3bound})
is shown in Fig \ref{fig:g3}.
\begin{figure}
\centering{}\includegraphics[scale=0.6]{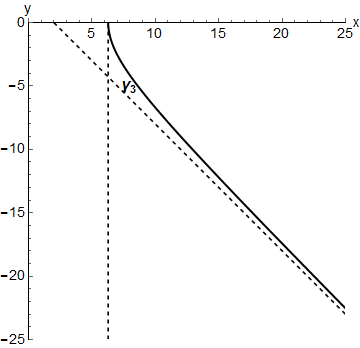}\protect\caption{\label{fig:g3}The curve $\gamma_{3}$ and its bounds $y+x=2,\quad x=2\pi$.}
\end{figure}
\end{proof}

\begin{lemma}
\label{lem:g23}For any $y\in(-\infty,0)$, we have $x_{2}(y)<x_{3}(y)$.\end{lemma}
\begin{proof}
It follows from Lemmas \ref{lem:g2} and \ref{lem:g3} that $x_{2}(y)<-y<2-y<x_{3}(y),\quad y\in(-\infty,0)$.
\hfill$\square$
\end{proof}

Lemmas \ref{lem:g1}--\ref{lem:g23} allow us to define the following
domains in the plane $Q\subset\mathbb{R}_{x,y}^{2}$:
\begin{eqnarray*}
m_{1} & = & \left\{ (x,y)\in\mathbb{R}^{2}\quad\mid\quad y<0,\quad0<x<x_{2}(y)\right\} ,\\
m_{2} & = & \left\{ (x,y)\in\mathbb{R}^{2}\quad\mid\quad y<0,\quad x_{2}(y)<x<x_{3}(y)\right\} ,\\
m_{3} & = & \left\{ (x,y)\in\mathbb{R}^{2}\quad\mid\quad y<0,\quad x_{3}(y)<x\right\} ,
\end{eqnarray*}
see Figure \ref{fig:g1g5}.
\begin{lemma}
\label{lem:m13}The domains $m_{1},m_{2},m_{3}\subset\mathbb{R}_{x,y}^{2}$
are open, connected and simply connected, with the following boundaries:
\begin{eqnarray*}
\partial m_{1} & = & \gamma_{1}\cup\gamma_{2}\cup\{O\},\\
\partial m_{2} & = & \gamma_{2}\cup\gamma_{3}\cup\gamma_{4}\cup\{O,P\},\\
\partial m_{3} & = & \gamma_{3}\cup\gamma_{5}\cup\{P\}.
\end{eqnarray*}
 Moreover, the quadrant $Q$ has the following decomposition into
disjoint subsets:
\[
Q=\left(\cup_{i=1}^{3}m_{i}\right)\cup\left(\cup_{i=1}^{5}\gamma_{i}\right)\cup\{O,P\}.
\]
\end{lemma}
\begin{proof}
Follows from the definition of the domains $m_{i}$ and from Lemmas
\ref{lem:g1}--\ref{lem:g23}. \hfill$\square$
\end{proof}

Define the inverse images of the sets $m_{i},\gamma_{i},$ and $P$
via the projection $p$ (\ref{eq:P}):
\begin{eqnarray*}
M_{9}^{\prime}=p^{-1}(m_{1}), & \quad M_{35}^{\prime}=p^{-1}(m_{2}), & \quad M_{1}^{\prime}=p^{-1}(m_{3}),\\
M_{29}^{\prime}=p^{-1}(\gamma_{1}), & \quad M_{25}^{\prime}=p^{-1}(\gamma_{2}), & \quad M_{21}^{\prime}=p^{-1}(\gamma_{3}),\\
M_{39}^{\prime}=p^{-1}(\gamma_{4}), & \quad M_{17}^{\prime}=p^{-1}(\gamma_{5}), & \quad M_{33}^{\prime}=p^{-1}(P).
\end{eqnarray*}
Explicitly, these sets are defined in Table \ref{tab:M_j}.
\begin{table}
\centering{}%
\begin{tabular}{|c|c|c|c|}
\hline 
$j$ & $y$ & $x$ & $z$\tabularnewline
\hline 
\hline 
1 & $(-\infty,0)$ & $(x_{3}(y),+\infty)$ & 0\tabularnewline
\hline 
9 & $(-\infty,0)$ & $(0,x_{2}(y))$ & 0\tabularnewline
\hline 
17 & 0 & $(2\pi,+\infty)$ & 0\tabularnewline
\hline 
21 & $(-\infty,0)$ & $x_{3}(y)$ & 0\tabularnewline
\hline 
25 & $(-\infty,0)$ & $x_{2}(y)$ & 0\tabularnewline
\hline 
29 & $(-\infty,0)$ & 0 & 0\tabularnewline
\hline 
33 & 0 & $2\pi$ & 0\tabularnewline
\hline 
35 & $(-\infty,0)$ & $(x_{2}(y),x_{3}(y))$ & 0\tabularnewline
\hline 
39 & 0 & $(0,2\pi)$ & 0\tabularnewline
\hline 
\end{tabular}\protect\caption{\label{tab:M_j}Definition of $M_{j}^{\prime}\subset p^{-1}(Q).$}
\end{table}

Now we aim to prove that all the mappings $\mathrm{Exp}:N_{j}^{\prime}\rightarrow M_{j}^{\prime}$
are diffeomorphisms for the sets $N_{j}^{\prime}$ and $M_{j}^{\prime}$
defined by Tables \ref{tab:Nj}, \ref{tab:N35}, \ref{tab:N33_39},
\ref{tab:M_j}.
\begin{lemma}
\label{lem:Exp_curve}For any $j\in\left\{ 17,21,25,29,33,39\right\} $
the mapping $\mathrm{Exp}:N_{j}^{\prime}\rightarrow M_{j}^{\prime}$
is a diffeomorphism.\end{lemma}
\begin{proof}
Follows immediately from above lemmas:\end{proof}

\begin{itemize}
\item Lemma \ref{lem:g5} for $j=17$,
\item Lemma \ref{lem:g3} for $j=21$,
\item Lemma \ref{lem:g2} for $j=25$,
\item Lemma \ref{lem:g1} for $j=29$,
\item Lemma \ref{lem:g4} for $j=39$,
\item and it is obvious for $j=33$.\hfill$\square$
\end{itemize}
Now we consider the mappings of 2-dimensional domains.
\begin{lemma}
\label{lem:ExpN9}The mapping $\mathrm{Exp}:N_{9}^{\prime}\rightarrow M_{9}^{\prime}$
is a diffeomorphism.\end{lemma}
\begin{proof}
In the coordinates $p=\frac{t}{2k}$ and $\tau=\left(\varphi+\frac{t}{2}\right)/k,$
the domain $N_{9}^{\prime}$ is given as follows:
\[
N_{9}^{\prime}:\lambda\in C_{2}^{+},\quad s_{2}=0,\quad p=2K(k),\quad\tau\in(0,K(k)),\quad k\in(0,1).
\]
Introduce further the coordinate $u=\mathrm{am}(\tau)$, then, 
\[
N_{9}^{\prime}:s_{2}=0,\quad p=2K(k),\quad u\in\left(0,\frac{\pi}{2}\right),\quad k\in(0,1).
\]
In these coordinates the exponential mapping $\mathrm{Exp}(\lambda,t)=(x,y,z)$
is given as follows:
\begin{eqnarray*}
x & = & x_{9}(u,k)=\frac{4ka(k)\cos(u)}{1-k^{2}},\\
y & = & y_{9}(u,k)=-\frac{4a(k)\sqrt{1-k^{2}\sin^{2}(u)}}{1-k^{2}},\\
z & = & 0.
\end{eqnarray*}
Consider the mapping:
\begin{eqnarray*}
f_{9}:D_{u,k} & \rightarrow & \mathbb{R}_{x,y}^{2},\quad(u,k)\mapsto(x_{9},y_{9}),\\
D_{u,k} & = & \left(0,\frac{\pi}{2}\right)_{u}\times(0,1)_{k}.
\end{eqnarray*}
We have to show that the mapping $f_{9}:D\rightarrow m_{1}$ is a diffeomorphism. \end{proof}
\begin{enumerate}
\item[(1)] First we show that $f_{9}(D)\subset m_{1}$. \\
We fix any $k\in(0,1)$ and show that the curve $\Gamma:u\rightarrow(x_{9},y_{9}),\quad u\in\left(0,\frac{\pi}{2}\right),$
is contained in $m_{1}$. Compute first the boundary points of $\Gamma$:
\begin{eqnarray*}
u & \rightarrow & 0\implies\Gamma(u)\rightarrow(x_{2}(k),y_{2}(k))\in\gamma_{2},\\
u & \rightarrow & \frac{\pi}{2}\implies\Gamma(u)\rightarrow(0,y_{2}(k))\in\gamma_{1}.
\end{eqnarray*}
Further, since
\begin{eqnarray*}
\frac{\partial x_{9}}{\partial u} & = & -\frac{4ka(k)}{1-k^{2}}\sin(u)<0,\\
\frac{\partial y_{9}}{\partial u} & = & \frac{4k^{2}a(k)}{1-k^{2}}\frac{\sin(u)\cos(u)}{\sqrt{1-k^{2}\sin^{2}(u)}}>0,
\end{eqnarray*}
then the curve $\Gamma$ is a graph of the smooth function $x\mapsto y_{9}(x)$.
Since
\[
\frac{dy_{9}}{dx}=\frac{\partial y_{9}/\partial u}{\partial x_{9}/\partial u}=-\frac{k\cos(u)}{\sqrt{1-k^{2}\sin^{2}(u)}},\quad\textrm{for }u\in\left(0,\frac{\pi}{2}\right),
\]
then the curve $\Gamma$ is concave. Moreover,
\[
\left.\frac{dy_{9}}{dx}\right|_{u=0}=-k>\alpha(k)=\frac{dy_{2}}{dx},
\]
where $\alpha(k)$ is given by (\ref{eq:alpha(k)}). Since the curve
$\gamma_{2}$ is convex, it follows that the curve $\Gamma$ lies
below the curve $\gamma_{2}$. Thus $\Gamma\subset m_{1}.$ Consequently,
$f_{9}(D)\subset m_{1}$.
\item[(2)] Since
\begin{equation}
\frac{\partial(x_{9},y_{9})}{\partial(u,k)}=\frac{16k^{2}E(k)a(k)\sin(u)}{\left(1-k^{2}\right)^{2}\sqrt{1-k^{2}\sin^{2}(u)}}>0,
\end{equation}
then the mapping $f_{9}:D\rightarrow m_{1}$ is non-degenerate.
\item[(3)] Finally we show that the mapping $f_{9}:D\rightarrow m_{1}$ is proper.
\\
It is obvious that a sequence $(u_{n},k_{n})\rightarrow\partial D$
iff it has a subsequence on which at least one of the conditions hold:
\begin{equation}
u\rightarrow0,\quad u\rightarrow\frac{\pi}{2},\quad k\rightarrow0,\quad k\rightarrow1.\label{eq:u0pi/2}
\end{equation}
On the other hand, a sequence $(x_{n},y_{n})\rightarrow\partial m_{1}$
iff it has a subsequence on which at least one of the conditions hold:
\begin{equation}
x\rightarrow0,\quad x\rightarrow+\infty,\quad y\rightarrow0,\quad y\rightarrow-\infty,\quad x_{2}(y)-x\rightarrow0.\label{eq:x0+infty}
\end{equation}
We show that in each of the cases (\ref{eq:u0pi/2}) we have one of
the cases (\ref{eq:x0+infty}). If $k\rightarrow0$, then $x_{9}\rightarrow0$
and $y_{9}\rightarrow0$. We can assume below that $k\rightarrow\bar{k}\in(0,1]$.
\\
Let $\bar{k}\in(0,1)$. If $u\rightarrow0$, then $\left(x_{9},y_{9}\right)\rightarrow\left(x_{2}(k),y_{2}(k)\right)\in\gamma_{2}$
thus $x_{2}(y)-x\rightarrow0$. If $u\rightarrow\frac{\pi}{2}$, then
$x_{9}\rightarrow0$. Let $\bar{k}=1$. If $u\rightarrow0$, then
$x_{9}\rightarrow\infty$. If $u\rightarrow\frac{\pi}{2}$, then $y_{9}\rightarrow\infty$.\\
We proved that the mapping $f_{9}:D\rightarrow m_{1}$ is proper. 
\item[(4)] The sets $D,\, m_{1}\subset\mathbb{R}^{2}$ are open, connected and
simply connected. \\
Thus $f_{9}:D\rightarrow m_{1}$ is a diffeomorphism, as well as $\mathrm{Exp}:N_{9}^{\prime}\rightarrow M_{9}^{\prime}$.\hfill$\square$
\end{enumerate}
\begin{lemma}
\label{lem:ExpN1}The mapping $\mathrm{Exp}:N_{1}^{\prime}\rightarrow M_{1}^{\prime}$
is a diffeomorphism.\end{lemma}
\begin{proof}
In the coordinates $p=\frac{t}{2}$ and $\tau=\varphi+\frac{t}{2},$
the domain $N_{1}^{\prime}$ is given as follows:
\[
N_{1}^{\prime}:\lambda\in C_{1}^{0},\quad s_{1}=0,\quad p=2K(k),\quad\tau\in(0,K(k)),\quad k\in(0,1).
\]
Introduce further the coordinate $u=\mathrm{am}(\tau)$, then 
\[
N_{1}^{\prime}:s_{1}=0,\quad p=2K(k),\quad u\in\left(0,\frac{\pi}{2}\right),\quad k\in(0,1).
\]
In these coordinates the exponential mapping $\mathrm{Exp}(\lambda,t)=(x,y,z)$
is given as follows:
\begin{eqnarray*}
x & = & x_{1}(u,k)=\frac{4E(k)\sqrt{1-k^{2}\sin^{2}(u)}}{1-k^{2}},\\
y & = & y_{1}(u,k)=-\frac{4k\, E(k)\cos(u)}{1-k^{2}},\\
z & = & 0.
\end{eqnarray*}
Consider the mapping:
\begin{eqnarray*}
f_{1}:D_{u,k} & \rightarrow & \mathbb{R}_{x,y}^{2},\quad(u,k)\mapsto(x_{1},y_{1}),\\
D_{u,k} & = & \left(0,\frac{\pi}{2}\right)_{u}\times(0,1)_{k}.
\end{eqnarray*}
We have to show that the mapping $f_{1}:D\rightarrow m_{3}$ is a
diffeomorphism. \end{proof}

\begin{enumerate}
\item[(1)] First we show that $f_{1}(D)\subset m_{3}$. \\
If $(u,k)\in D$, then $x_{1}(u,k)>0,\quad y_{1}(u,k)<0$, thus $f_{1}(D)\subset\mathbb{R}_{+-}^{2}=\left\{ (x,y)\in\mathbb{R}^{2}\quad\mid\quad x>0,\quad y<0\right\} $.
The boundary of the domain $m_{3}$ in $\mathbb{R}_{+-}^{2}$ is the
curve $\gamma_{3}$ and along this curve we have $\frac{y_{4}(k)}{x_{4}(k)}=-k$.
Thus
\[
\gamma_{3}=\left\{ (x,y)\in\mathbb{R}_{+-}^{2}\quad\mid\quad x=\frac{4E\left(-\frac{y}{x}\right)}{1-\frac{y^{2}}{x^{2}}}\right\} ,
\]
so
\[
m_{3}=\left\{ (x,y)\in\mathbb{R}_{+-}^{2}\quad\mid\quad x>\frac{4E\left(-\frac{y}{x}\right)}{1-\frac{y^{2}}{x^{2}}}\right\} .
\]
Consider the function 
\[
\varphi_{1}(u,k)=\left.x-\frac{4E\left(-\frac{y}{x}\right)}{1-\frac{y^{2}}{x^{2}}}\right|_{x=x_{1}(u,k),\, y=y_{1}(u,k)}.
\]
We have to show that $\varphi_{1}(u,k)>0$ for $(u,k)\in D$. Since
\begin{eqnarray*}
\varphi_{1}(u,k) & = & \frac{4E(k)\sqrt{1-k^{2}\sin^{2}(u)}}{1-k^{2}}-\frac{4E(\bar{k})}{1-\frac{k^{2}\cos^{2}u}{1-k^{2}\sin^{2}u}}\\
 & = & \frac{4\sqrt{1-k^{2}\sin^{2}(u)}}{1-k^{2}}\left(E(k)-E(\bar{k})\sqrt{1-k^{2}\sin^{2}(u)}\right),
\end{eqnarray*}
where $\bar{k}=\frac{k\cos(u)}{\sqrt{1-k^{2}\sin^{2}u}}$, we have
to show that
\[
\varphi_{2}(u,k)=E(k)-E(\bar{k})\sqrt{1-k^{2}\sin^{2}(u)}>0,\quad(u,k)\in D.
\]
Since $\varphi_{2}(0,k)=0$ and
\[
\frac{\partial\varphi_{2}}{\partial u}=\frac{\tan(u)}{\sqrt{1-k^{2}\sin^{2}(u)}}\varphi_{3}(u,k),
\]
where $\varphi_{3}(u,k)=\left(1-k^{2}\sin^{2}(u)\right)E(\bar{k})-\left(1-k^{2}\right)K(\bar{k})$,
it is sufficient to show that $\varphi_{3}(u,k)>0$ for all $(u,k)\in D$.
By Lemma \ref{lem:a(k)}, we have
\[
a(k)=E(k)-\left(1-k^{2}\right)K(k)>0,\quad k\in(0,1),
\]
thus
\begin{eqnarray*}
a(\bar{k}) & = & E(\bar{k})-\left(1-\bar{k}^{2}\right)K(\bar{k})\\
 & = & \frac{\left(1-k^{2}\sin^{2}(u)\right)E(\bar{k})-\left(1-k^{2}\right)K(\bar{k})}{1-k^{2}\sin^{2}(u)}>0.
\end{eqnarray*}
That is, $\varphi_{3}(u,k)>0,\quad\forall(u,k)\in D$. Thus it follows
that $f_{1}(D)\subset m_{3}$, i.e., $\mathrm{Exp}(N_{1}^{\prime})\subset M_{1}^{\prime}$. 
\item[(2)] Since
\[
\frac{\partial(x_{1},y_{1})}{\partial(u,k)}=-\frac{16E(k)\, a(k)\sin(u)}{\left(1-k^{2}\right)^{2}\sqrt{1-k^{2}\sin^{2}(u)}}<0,
\]
then the mapping $f_{1}:D\rightarrow m_{3}$ is non-degenerate.
\item[(3)] Finally we show that the mapping $f_{1}:D\rightarrow m_{3}$ is proper.
\\
In order to show that the mapping $f_{1}:D\rightarrow m_{3}$ is proper,
we show that if a sequence $(u_{n},k_{n})\in D$ satisfies one of
the conditions: 
\[
u\rightarrow0,\quad u\rightarrow\frac{\pi}{2},\quad k\rightarrow0,\quad k\rightarrow1,
\]
then its image $(x_{n},y_{n})=f_{1}(u_{n},k_{n})$ satisfies one of
the conditions: 
\[
x\rightarrow0,\quad x\rightarrow+\infty,\quad y\rightarrow0,\quad y\rightarrow\infty,\quad x_{3}(y)-x\rightarrow0.
\]
 We can assume that $k\rightarrow\bar{k}\in(0,1],\quad u\in\bar{u}\in[0,\frac{\pi}{2}]$.
If $\bar{k}=0$, then $y_{1}\rightarrow0$. \\
Let $\bar{k}\in(0,1)$. If $\bar{u}\rightarrow0$, then $\left(x_{1},y_{1}\right)\rightarrow\left(x_{3}(k),y_{3}(k)\right)\in\gamma_{3}$,
thus $x_{3}(y)-x\rightarrow0$. If $\bar{u}=\frac{\pi}{2}$, then
$y_{1}\rightarrow0$. Let $\bar{k}=1$. If $\bar{u}\in[0,\frac{\pi}{2})$,
then $x_{1}\rightarrow\infty$, $y_{1}\rightarrow\infty$. Let $\bar{u}=\frac{\pi}{2}$,
then 
\begin{eqnarray*}
y_{1} & \sim & -\frac{4\cos(u)}{1-k^{2}},\\
x_{1} & \sim & 4\sqrt{\frac{1}{1-k^{2}}+k^{2}\left(\frac{\cos(u)}{1-k^{2}}\right)^{2}}.
\end{eqnarray*}
We can assume that $\frac{\cos(u)}{1-k^{2}}\rightarrow d\in[0,+\infty)$.
If $d\in[0,+\infty)$, then $x_{1}\rightarrow+\infty$, and if $d=+\infty$,
then $y_{1}\rightarrow\infty$. \\
We proved that the mapping $f_{1}:D\rightarrow m_{3}$ is proper.
\item[(4)] The sets $D,\, m_{3}\subset\mathbb{R}^{2}$ are open, connected and
simply connected. 
\end{enumerate}
Thus $f_{1}:D\rightarrow m_{3}$ is a diffeomorphism, as well as the
mapping $\mathrm{Exp}:N_{1}^{\prime}\rightarrow M_{1}^{\prime}$.\hfill$\square$

\begin{lemma}
\label{lem:ExpN35}The mapping $\mathrm{Exp}:N_{35}^{\prime}\rightarrow M_{35}^{\prime}$
is a diffeomorphism.\end{lemma}
\begin{proof}
It follows from Tables \ref{tab:N35}, \ref{tab:M_j} that
\begin{eqnarray*}
N_{35}^{\prime} & = & \left\{ (\lambda,t)\in N\quad\mid\quad\gamma_{\frac{t}{2}}=0,\quad c_{\frac{t}{2}}>0,\quad t\in(0,\mathbf{t}(\lambda))\right\} ,\\
M_{35}^{\prime} & = & \left\{ q\in M\quad\mid\quad z=0,\quad y<0,\quad x_{2}(y)<x<x_{3}(y)\right\} .
\end{eqnarray*}
Further we have an obvious decomposition
\begin{eqnarray*}
N_{35}^{\prime} & = & N_{35,1}^{\prime}\sqcup N_{35,2}^{\prime}\sqcup N_{35,3}^{\prime},\\
N_{35,j}^{\prime} & = & N_{35}^{\prime}\cap N_{j},\quad j=1,2,3.
\end{eqnarray*}
\end{proof}

\begin{enumerate}
\item[(1)] We show first that $\mathrm{Exp}(N_{35}^{\prime})\subset M_{35}^{\prime}$.
\\
Consider the set $N_{35,2}^{\prime}$. In the coordinates $p=\frac{t}{2k}$
and $\tau=\left(\varphi+\frac{t}{2}\right)/k,$ the domain $N_{35,2}^{\prime}$
is given as follows:
\[
N_{35,2}^{\prime}:\lambda\in C_{2}^{+},\quad s_{2}=1,\quad p=(0,2K(k)),\quad\tau=0,\quad k\in(0,1).
\]
Introduce further the coordinate $u=\mathrm{am}(p)$, then 
\[
N_{35,2}^{\prime}:\lambda\in C_{2}^{+},\quad s_{2}=1,\quad u=(0,2\pi),\quad\tau=0,\quad k\in(0,1).
\]
In these coordinates the exponential mapping $\mathrm{Exp}(\lambda,t)=(x,y,z),\quad(\lambda,t)\in N_{35,2}^{\prime}$
is given as follows:
\begin{eqnarray*}
x & = & x_{35}(u,k)=\frac{2k}{1-k^{2}}\left[\sin(u)\sqrt{1-k^{2}\sin^{2}(u)}-\cos(u)\,\alpha(u,k)\right],\\
y & = & y_{35}(u,k)=-\frac{2}{1-k^{2}}\left[\sqrt{1-k^{2}\sin^{2}(u)}\,\alpha(u,k)-k^{2}\sin(u)\cos(u)\right],\\
z & = & 0,
\end{eqnarray*}
where $\alpha(u,k)=E(u,k)-\left(1-k^{2}\right)F(u,k).$ Thus $\mathrm{Exp}(N_{35,2}^{\prime})\subset\left\{ q\in M\quad\mid\quad z=0\right\} $.
Now we show that $x_{35}(u,k)>0,\quad y_{35}(u,k)<0$ for $(u,k)\in(0,\frac{\pi}{2})\times(0,1)$.
We have to prove the double inequality
\begin{eqnarray*}
\alpha_{1}(u,k) & < & \alpha(u,k)<\alpha_{2}(u,k),\quad(u,k)\in(0,\frac{\pi}{2})\times(0,1),\\
\alpha_{1}(u,k) & = & \frac{k^{2}\sin(u)\cos(u)}{\sqrt{1-k^{2}\sin^{2}(u)}},\\
\alpha_{2}(u,k) & = & \frac{\sin(u)\sqrt{1-k^{2}\sin^{2}(u)}}{\cos(u)}.
\end{eqnarray*}
This double inequality follows since
\begin{eqnarray*}
\alpha_{1}(0,k) & = & \alpha(0,k)=\alpha_{2}(0,k)=0,\\
\frac{\partial}{\partial u}\left(\alpha(u,k)-\alpha_{1}(u,k)\right) & = & \left(1-k^{2}\right)\sin^{2}(u)>0,\\
\frac{\partial}{\partial u}\left(\alpha_{2}(u,k)-\alpha(u,k)\right) & = & 1-k^{2}>0.
\end{eqnarray*}
Thus $x_{35}(u,k)>0,\quad y_{35}(u,k)<0$ for $(u,k)\in\left(0,\frac{\pi}{2}\right)\times(0,1)$.
If $u\in[\frac{\pi}{2},\pi),\quad k\in(0,1)$, then $\sin(u)>0,\quad\cos(u)\leq0,\quad\alpha(u,k)>0$,
thus $x_{35}(u,k)>0,\quad y_{35}(u,k)<0$. We proved that \\
$\mathrm{Exp}(N_{35,2}^{\prime})\subset\left\{ q\in M\quad\mid\quad z=0,\quad x>0,\quad y<0\right\} $.
The sets $N_{35,1}^{\prime}$ and $N_{35,3}^{\prime}$ are considered
similarly. Thus it follows that
\[
\mathrm{Exp}(N_{35}^{\prime})\subset\mathbb{R}_{+-}^{2}:=\left\{ q\in M\quad\mid\quad z=0,\quad x>0,\quad y<0\right\} .
\]
We now show that $\mathrm{Exp}(N_{35}^{\prime})\subset M_{35}^{\prime}$.
Notice the decomposition
\[
\mathbb{R}_{+-}^{2}=M_{1}^{\prime}\sqcup M_{9}^{\prime}\sqcup M_{21}^{\prime}\sqcup M_{25}^{\prime}\sqcup M_{35}^{\prime}.
\]
By contradiction, let $\mathrm{Exp}(N_{35}^{\prime})\not\subset M_{35}^{\prime}$,
then $\mathrm{Exp}(N_{35}^{\prime})\cap\left(M_{1}^{\prime}\sqcup M_{9}^{\prime}\sqcup M_{21}^{\prime}\sqcup M_{25}^{\prime}\right)\ni q$.
Let $q\in\mathrm{Exp}(N_{35}^{\prime})\cap M_{1}^{\prime}$ (the cases
of intersection with $M_{9}^{\prime},M_{21}^{\prime},M_{25}^{\prime}$
are considered similarly). Then there exist $\left(\lambda_{35},t_{35}\right)\in N_{35}^{\prime}$,
$\left(\lambda_{1},t_{1}\right)\in N_{1}^{\prime}$ such that $q=\mathrm{Exp}\left(\lambda_{35},t_{35}\right)=\mathrm{Exp}\left(\lambda_{1},t_{1}\right)$.
Notice that
\begin{eqnarray}
\left(\lambda_{35},t_{35}\right) & \in & N_{35}^{\prime}\implies t_{35}<t_{\mathrm{cut}}\left(\lambda_{35}\right),\label{eq:lam35}\\
\left(\lambda_{1},t_{1}\right) & \in & N_{1}^{\prime}\implies t_{1}<t_{\mathrm{cut}}\left(\lambda_{1}\right).\label{eq:lam1}
\end{eqnarray}
If $t_{35}<t_{1}$, then the trajectory $\mathrm{Exp}\left(\lambda_{1},t\right),\quad t\in[0,t_{1}],$
is not optimal which contradicts to (\ref{eq:lam1}) . If $t_{35}\geq t_{1}$,
then the trajectory $\mathrm{Exp}(\lambda_{35},t),\quad t\in[0,t_{35}+\varepsilon]$
is not optimal for small $\varepsilon>0$ which contradicts to (\ref{eq:lam35}).
Thus $\mathrm{Exp}(N_{35}^{\prime})\cap M_{1}^{\prime}=\emptyset$.
Then it follows that $\mathrm{Exp}(N_{35}^{\prime})\subset M_{35}^{\prime}$. 
\item[(2)] We now prove that $\mathrm{Exp}:N_{35}^{\prime}\rightarrow M_{35}^{\prime}$
is non-degenerate.\\
Let $\nu=(\lambda,t)\in N_{35,2}^{\prime}$. In the coordinates $(p,\tau,k)$
on $N_{35,2}^{\prime}$, we have $p\in(0,2K(k)),\quad\tau=0,\quad k\in(0,1)$.
Since $t<4K(k)=t_{\mathrm{cut}}(\lambda)\leq t_{1}^{\mathrm{conj}}(\lambda)$,
therefore the Jacobian $\frac{\partial q}{\partial\nu}(\nu)\neq0$.
We have
\[
\frac{\partial q}{\partial\nu}=\frac{\partial(x,y,z)}{\partial(p,\tau,k)}=\left|\begin{array}{ccc}
x_{p} & x_{\tau} & x_{k}\\
y_{p} & y_{\tau} & y_{k}\\
z_{p} & z_{\tau} & z_{k}
\end{array}\right|.
\]
Since $\mathrm{Exp}\left(N_{i,2}^{\prime}\right)\subset\left\{ q\in M\quad\mid\quad z=0\right\} ,$
then $z_{p}(\nu)=z_{k}(\nu)=0$, thus
\[
\frac{\partial q}{\partial\nu}(\nu)=\frac{\partial(x,y)}{\partial(p,k)}(\nu)\, z_{\tau}(\nu)\neq0,
\]
so $\frac{\partial(x,y)}{\partial(p,k)}(\nu)\neq0$. Since $\nu\in N_{35,2}^{\prime}$
is arbitrary, then $\left.\mathrm{Exp}\right|_{N_{35,2}^{\prime}}$
is non-degenerate. Similarly it follows that $\mathrm{Exp}$ is non-degenerate
at any point $\nu\in N_{35,1}^{\prime}\cup N_{35,3}^{\prime}$. 
\item[(3)] The mapping $\mathrm{Exp}:N_{35}^{\prime}\rightarrow M_{35}^{\prime}$
is proper. This follows similarly to the proof of properness of $\mathrm{Exp}:D_{1}\rightarrow M_{1}$.
\item[(4)] It is obvious that $M_{35}^{\prime}$ is a connected, simply connected
2-dimensional manifold. In order to prove the same property for $N_{35}^{\prime}$,
consider the vector field
\[
\overrightarrow{P}=c\frac{\partial}{\partial\gamma}-\sin\gamma\frac{\partial}{\partial c}\in\mathrm{Vec}(N).
\]
Since
\[
e^{t/2\overrightarrow{P}}\left(N_{35}^{\prime}\right)=\left\{ (\lambda,t)\in N\quad\mid\quad\gamma=0,\quad c>0,\quad t<\mathbf{t}(\lambda)\right\} 
\]
is a connected, simply connected 2-dimensional manifold, the same
properties hold for the set $N_{35}^{\prime}$. \\
Then it follows that $\mathrm{Exp}:N_{35}^{\prime}\rightarrow M_{35}^{\prime}$
is a diffeomorphism. \hfill$\square$
\end{enumerate}

\subsection{Stratification of the set $M^{\prime}$}

Define subsets $M_{j}^{\prime}\subset M^{\prime},\quad j=1,\ldots,40,$
as follows:
\begin{itemize}
\item For $j\in\left\{ 1,9,17,21,25,29,33,35,39\right\} ,$ the sets $M_{j}$
are given by Table \ref{tab:M_j},
\item For the rest $j$ the sets $M_{j}^{\prime}$ are given by equalities
(\ref{eq:epsiMj})--(\ref{eq:eps4Mj}): 
\begin{eqnarray}
\varepsilon^{i}\left(M_{j}^{\prime}\right) & = & M_{j+i}^{\prime},\quad i=1,\ldots,7,\quad j=1,9,\label{eq:epsiMj}\\
\varepsilon^{2i}\left(M_{17}^{\prime}\right) & = & M_{17+i}^{\prime},\quad i=1,2,3,\label{eq:eps2iMj}\\
\varepsilon^{2+i}\left(M_{j}^{\prime}\right) & = & M_{j+i}^{\prime},\quad i=1,2,3,\quad j=21,25,29,35,\label{eq:eps2+iMj}\\
\varepsilon^{4}\left(M_{j}^{\prime}\right) & = & M_{j+1}^{\prime},\quad j=33,39.\label{eq:eps4Mj}
\end{eqnarray}
\end{itemize}
\begin{lemma}
A stratification of $M^{\prime}$ is given as:
\begin{equation}
M^{\prime}=\sqcup_{j=1}^{40}M_{j}^{\prime}.\label{eq:M'Decomposn}
\end{equation}
\end{lemma}
\begin{proof}
Follows from Lemma \ref{lem:m13} and the description of the action
of reflections $\varepsilon^{i}$ in the plane $\left\{ z=0\right\} $,
see Table \ref{tab:epsiz_0}.

\begin{table}
\begin{centering}
\emph{}%
\begin{tabular}{|c|c|c|c|c|c|c|c|}
\hline 
$i$ & 1 & 2 & 3 & 4 & 5 & 6 & 7\tabularnewline
\hline 
\hline 
$x$ & $x$ & $x$ & $x$ & $-x$ & $-x$ & $-x$ & $-x$\tabularnewline
\hline 
$y$ & $-y$ & $y$ & $-y$ & $y$ & $-y$ & $y$ & $-y$\tabularnewline
\hline 
\end{tabular}\protect\caption{\label{tab:epsiz_0}Action of $\varepsilon^{i}$ in the plane $\left\{ z=0\right\} $}

\par\end{centering}

\end{table}
\hfill$\square$
\end{proof}
Stratification (\ref{eq:M'Decomposn}) is shown in Figure \ref{fig:M'decomp}.

\begin{figure}
\centering{}\includegraphics[scale=0.8]{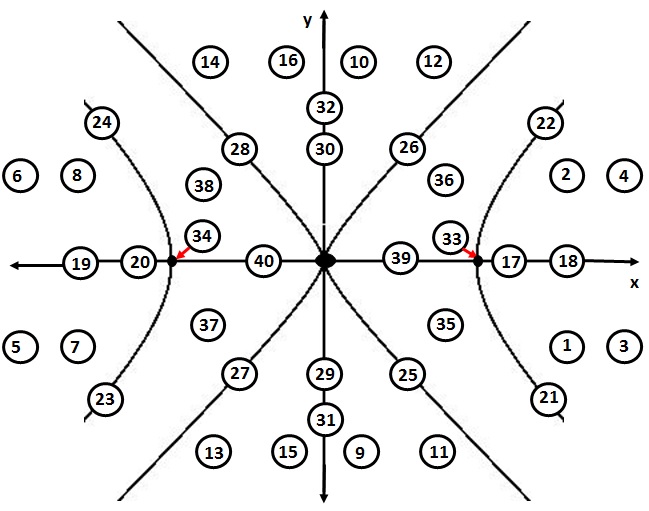}\protect\caption{\label{fig:M'decomp}Stratification of $M^{\prime}$}
\end{figure}

\begin{theorem}
\label{thm:expNi_Mi}For any $i=1,\ldots,40,$ the mapping $\mathrm{Exp}:N_{i}^{\prime}\rightarrow M_{i}^{\prime}$
is a diffeomorphism. \end{theorem}
\begin{proof}
Follows from Lemmas \ref{lem:Exp_curve}--\ref{lem:ExpN35} via the
symmetries $\varepsilon^{i}$ of the exponential mapping. \hfill$\square$
\end{proof}

Define the following important sets:
\begin{itemize}
\item the cut locus $\mathrm{Cut}=\left\{ \mathrm{Exp}(\lambda,t_{\mathrm{cut}}(\lambda))\quad\mid\quad\lambda\in C\right\} ,$
\item the first Maxwell set \\
$\mathrm{Max}=\left\{ q_{1}\in M\quad\mid\quad\exists\textrm{ minimizers }q^{\prime}(t)\not\equiv q^{\prime\prime}(t),\quad t\in[0,t_{1}],\textrm{ such that }q^{\prime}(t_{1})=q^{\prime\prime}(t_{1})=q_{1}\right\} .$
\item the first conjugate locus $\mathrm{Conj}=\left\{ \mathrm{Exp}(\lambda,t_{1}^{\mathrm{conj}}(\lambda))\quad\mid\quad\lambda\in C\right\} ,$
\item the rest of the points in $M^{\prime}$ compared with $\mathrm{Cut}$,
i.e., $\mathrm{Rest}=M^{\prime}\backslash\mathrm{Cut}$.
\end{itemize}
We have the following explicit description of these sets:
\begin{eqnarray*}
\mathrm{Cut} & = & \cup\left\{ M_{i}^{\prime}\quad\mid\quad i=1,\ldots,34\right\} ,\\
\mathrm{Max} & = & \cup\left\{ M_{i}^{\prime}\quad\mid\quad i=1,\ldots,20,29,\ldots,32\right\} ,\\
\mathrm{Conj}\,\cap\,\mathrm{Cut} & = & \cup\left\{ M_{i}^{\prime}\quad\mid\quad i=21,\ldots,28,33,34\right\} ,\\
\mathrm{Rest} & = & \cup\left\{ M_{i}^{\prime}\quad\mid\quad i=35,\ldots,40\right\} ,
\end{eqnarray*}
Thus we get the following decomposition of the sets $M^{\prime}$:
\begin{eqnarray*}
M^{\prime} & = & \mathrm{Cut}\,\sqcup\,\mathrm{Rest},\\
\mathrm{Cut} & = & \mathrm{Max}\,\sqcup(\mathrm{Conj}\,\cap\,\mathrm{Cut}).
\end{eqnarray*}
The global structure of the cut locus is shown in Figure \ref{fig:cut_locus}.
\begin{figure}
\centering{}\includegraphics{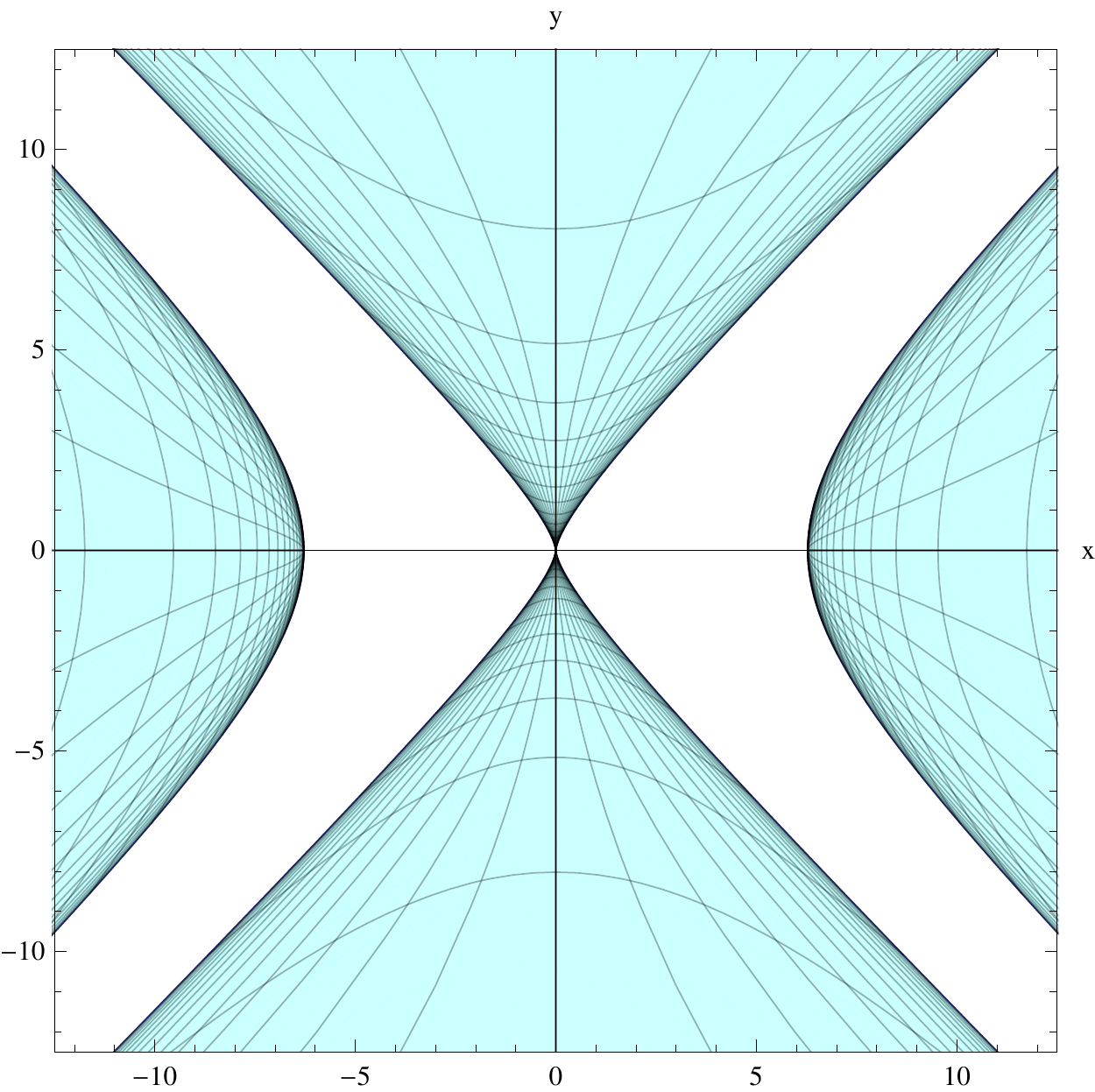}\protect\caption{\label{fig:cut_locus}Cut Locus}
\end{figure}
From our analysis of the exponential mapping, we get the following
description of the cut time and the optimal synthesis on $\mathrm{SH}(2)$.
\begin{theorem}
\label{thm:cut_time_exact}We have the following explicit description
of the cut time, $t_{\mathrm{cut}}(\lambda)=\mathbf{t}(\lambda)$
for any $\lambda\in C$. In detail: 
\begin{eqnarray*}
\lambda & \in & C_{1}\implies t_{\mathrm{cut}}(\lambda)=t_{1}^{\mathrm{Max}}(\lambda)=4K(k),\\
\lambda & \in & C_{2}\implies t_{\mathrm{cut}}(\lambda)=t_{1}^{\mathrm{Max}}(\lambda)=4kK(k),\\
\lambda & \in & C_{4}\implies t_{\mathrm{cut}}(\lambda)=t_{1}^{\mathrm{conj}}(\lambda)=2\pi,\\
\lambda & \in & C_{3}\cup C_{5}\implies t_{\mathrm{cut}}(\lambda)=+\infty.
\end{eqnarray*}
\end{theorem}
\begin{proof}
If $\lambda\in C\backslash C_{4}$, then we know from Theorem \ref{thm:Cut_time}
that $t_{\mathrm{cut}}(\lambda)=\mathbf{t}(\lambda)=t_{1}^{\mathrm{Max}}(\lambda)$.
It remains to consider the case $\lambda\in C_{4}^{0}\cup C_{4}^{1}$.
Let $\lambda\in C_{4}^{0}$, then $q_{t}=\mathrm{Exp}(\lambda,t)=(t,0,0).$
For any $t\in[0,t_{1}],\quad t_{1}=\mathbf{t}(\lambda)=2\pi$, the
point $q_{t}$ is connected with $q_{0}$ by a unique geodesic $\mathrm{Exp}(\lambda^{1},s),\quad s\in(0,s_{1}]$,
with $(\lambda^{1},s_{1})\in\widehat{N}$, namely $(\lambda^{1},s_{1})=(\lambda,t)\in N_{39}^{\prime}$
for $t\in(0,2\pi)$, and $(\lambda^{1},s_{1})=(\lambda,t)\in N_{33}^{\prime}$
for $t=2\pi$. Thus the geodesic $q_{t},\quad t\in[0,t_{1}]$ is a
minimizer.

It follows that $t_{\mathrm{cut}}(\lambda)=\mathbf{t}(\lambda)=t_{1}^{\mathrm{conj}}(\lambda)=2\pi$
for $\lambda\in C_{4}^{0}$. By applying a reflection $\varepsilon^{i}$,
we get a similar equality for $\lambda\in C_{4}^{1}$. \hfill$\square$
\end{proof}

From the above description of the structure of the exponential mapping,
we get the following statement.
\begin{theorem}
\label{thm:syn}\end{theorem}
\begin{enumerate}
\item For every point $q_{1}\in\widetilde{M}\cup\mathrm{Rest}$, there exists
a unique minimizer $q(t),\quad t\in[0,t_{1}]$, for which the endpoint
$q(t_{1})=q_{1}$ is neither a cut point nor a conjugate point. 
\item For any point $q_{1}\in\mathrm{Max}$, there exist exactly two minimizers
that connect $q_{0}$ to $q_{1}$ for which $q_{1}$ is a cut point
but not a conjugate point.
\item For any point $q_{1}\in\mathrm{Conj}\,\cap\,\mathrm{Cut}$, there
exists a unique minimizer that connects $q_{0}$ to $q_{1}$ for which
$q_{1}$ is both a cut and a conjugate point, but not a Maxwell point. 
\end{enumerate}

\section{Sub-Riemannian Caustics and Sphere}

In \cite{Max_Conj_SH2} we presented plots of sub-Riemannian sphere
and sub-Riemannian wavefront in the rectifying coordinates $(R_{1},R_{2},z)$.
Here we perform another graphic study of the essential sub-Riemannian
objects, i.e., sub-Riemannian caustic and sub-Riemannian sphere. Recall
that the sub-Riemannian caustic which is the first conjugate locus
is given as:
\[
\mathrm{Conj}=\left\{ \mathrm{Exp}\left(\lambda,t_{1}^{\mathrm{conj}}(\lambda)\right)\quad|\quad\lambda\in C\right\} .
\]
The caustic is presented in Figure \ref{fig:Caustic_C1}. The component
starting at $(0,0,0)$ is the local component of the caustic whereas
other two parts on right and left side are the parts of the global
component of the first caustic. The red colored surface inside the
local and global components of the caustic is the cut locus whereas
we see that the boundary of cut locus forms the boundary of the caustic.
A zoomed version of the local component of the caustic is separately
shown in Figure \ref{fig:Caustic_Local}. It is evident that it is
a four cusp surface as predicted in \cite{Agrchev_Barilari_Boscain_SR}.
A combined plot of first and second caustic is also shown in Figure
\ref{fig:Caustic_1_2}. Note that in the local component of the caustic,
the first caustic is solid and the second caustic is transparent whereas
in the global component of the caustic, the second caustic is solid
and the first caustic is transparent. 
\begin{figure}
\begin{centering}
\includegraphics[scale=0.8]{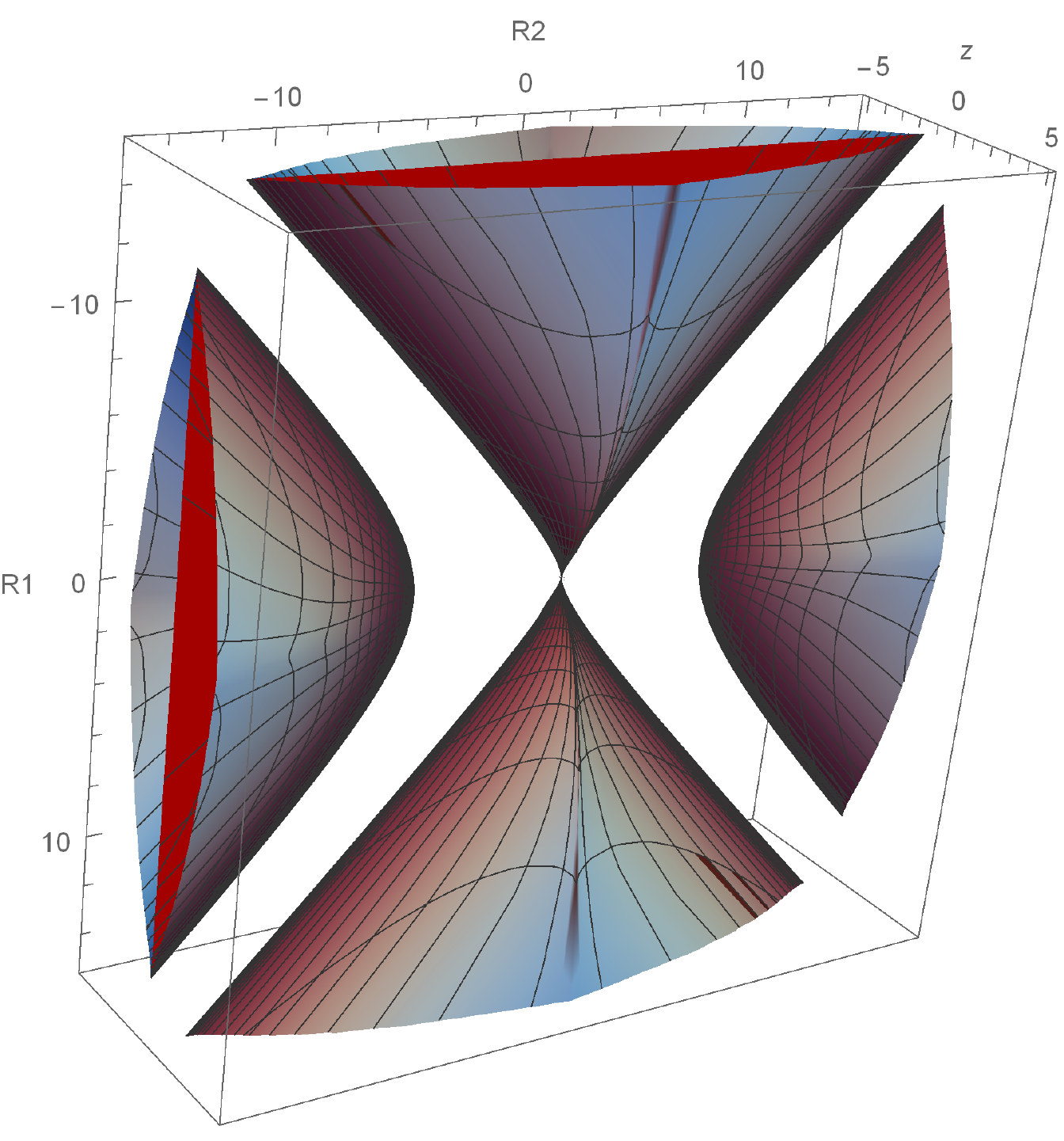}\protect\caption{\label{fig:Caustic_C1}Sub-Riemannian caustic and cut locus}

\par\end{centering}

\end{figure}
\begin{figure}
\centering{}\includegraphics[scale=0.8]{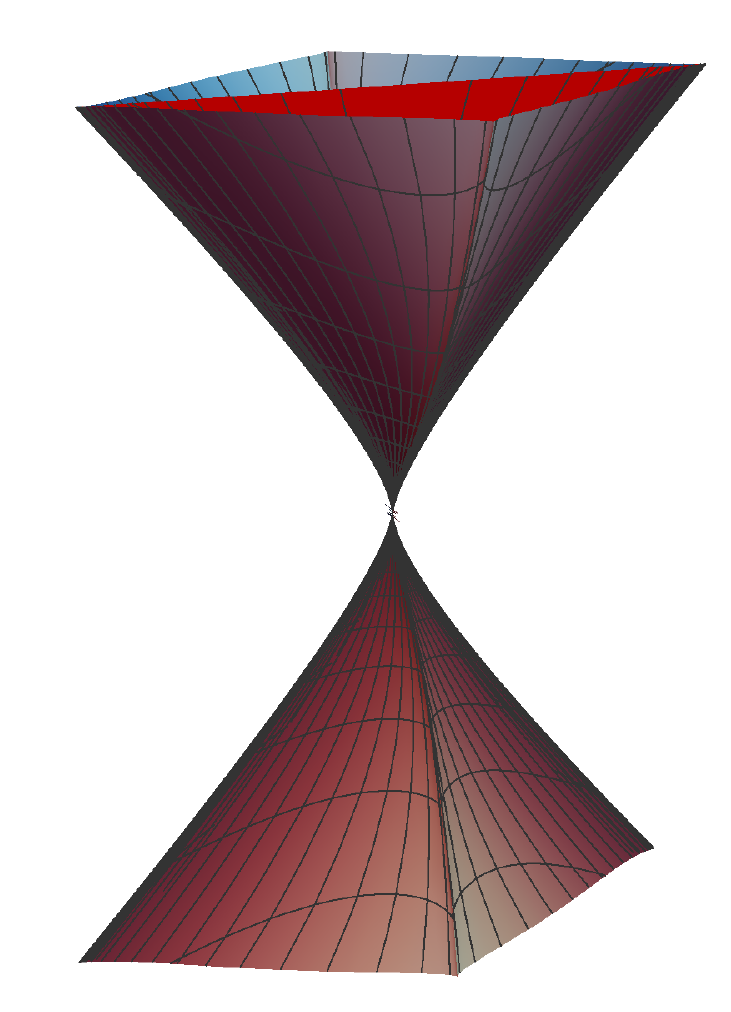}\protect\caption{\label{fig:Caustic_Local}Local component of sub-Riemannian caustic
and cut locus}
\end{figure}
\begin{figure}
\centering{}\includegraphics[scale=0.5]{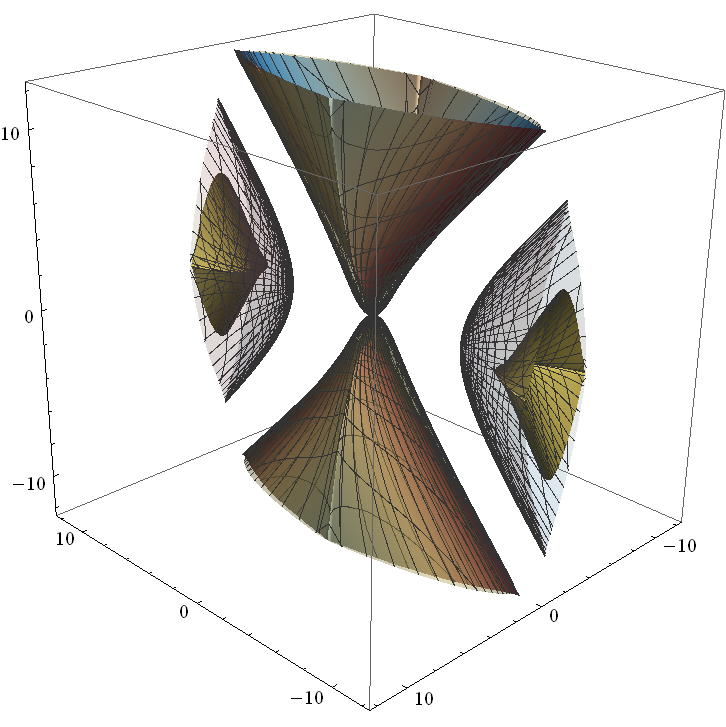}\protect\caption{\label{fig:Caustic_1_2}Sub-Riemannian first and second caustic}
\end{figure}
The sub-Riemannian sphere $S_{R}(q_{0};R)$ at $q_{0}$ is the set
of end-points of minimizing geodesics of sub-Riemannian length $R$
and starting from $q{}_{0}$:

\begin{eqnarray*}
S_{R} & = & \left\{ \mathrm{Exp}(\lambda,R)\in M\quad\vert\quad\lambda\in C,\quad t_{\mathrm{cut}}(\lambda)\geq R\right\} =\left\{ q\in M\quad\vert\quad d(q_{0},q)=R\right\} .
\end{eqnarray*}
The following plots are presented:
\begin{enumerate}
\item Sphere of radius $R=\pi$ (Figure \ref{fig:sphereRpiR12}),
\item Sphere of radius $R=2\pi$ (Figure \ref{fig:sphereR2piR12}),

\item Intersection of the cut locus with the hemisphere $z<0$ of radius
$R=\pi$ (Figure \ref{fig:cutnegsphereRpiR12}),
\item Intersection of the cut locus with the hemisphere $z<0$ of radius
$R=2\pi$ (Figure \ref{fig:cutnegsphereR2piR12}),
\item Intersection of the cut locus with the hemisphere $z<0$ of radius
$R=3\pi$ (Figure \ref{fig:cutnegsphereR3piR12}),
\item Matryoshka of hemispheres $z<0$ of radii $R=\pi$ and $R=2\pi$ (Figure
\ref{fig:matr2R12}).
\end{enumerate}
\begin{figure}
\centering{}
\begin{minipage}[t]{0.45\columnwidth}%
\begin{center}
\emph{\includegraphics[scale=0.3]{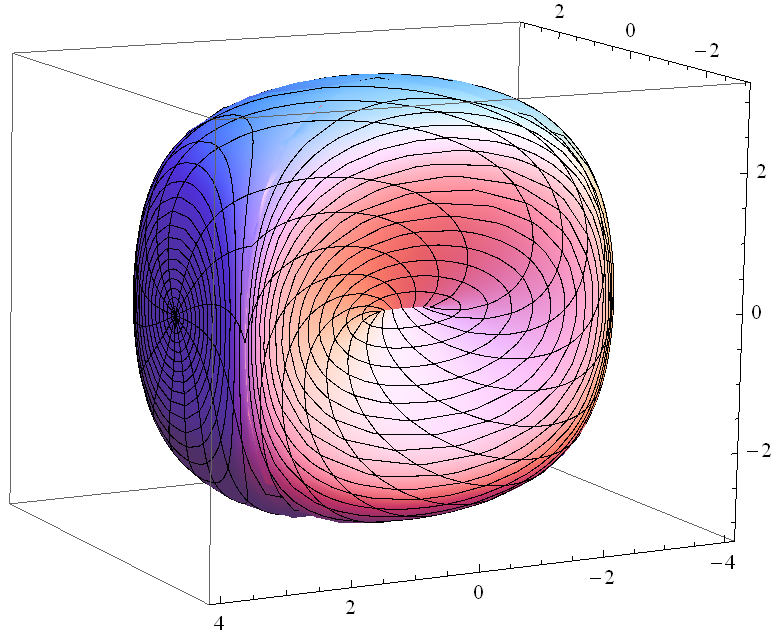}}
\par\end{center}

\protect\caption{\label{fig:sphereRpiR12}Sub-Riemannian sphere of radius $R=\pi$}
\end{minipage}\emph{\hfill{}}%
\begin{minipage}[t]{0.45\columnwidth}%
\begin{center}
\emph{\includegraphics[scale=0.3]{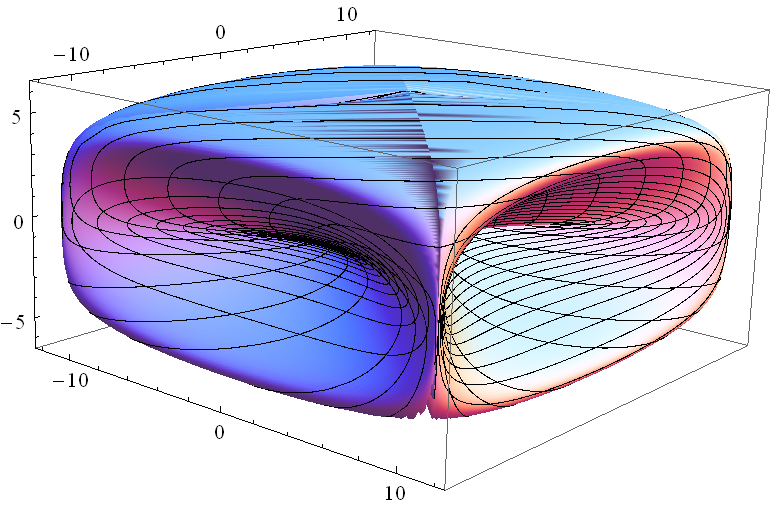}}
\par\end{center}

\protect\caption{\label{fig:sphereR2piR12}Sub-Riemannian sphere of radius $R=2\pi$}
\end{minipage}
\end{figure}

\begin{figure}
\begin{minipage}[t]{0.45\columnwidth}%
\begin{center}
\includegraphics[scale=0.3]{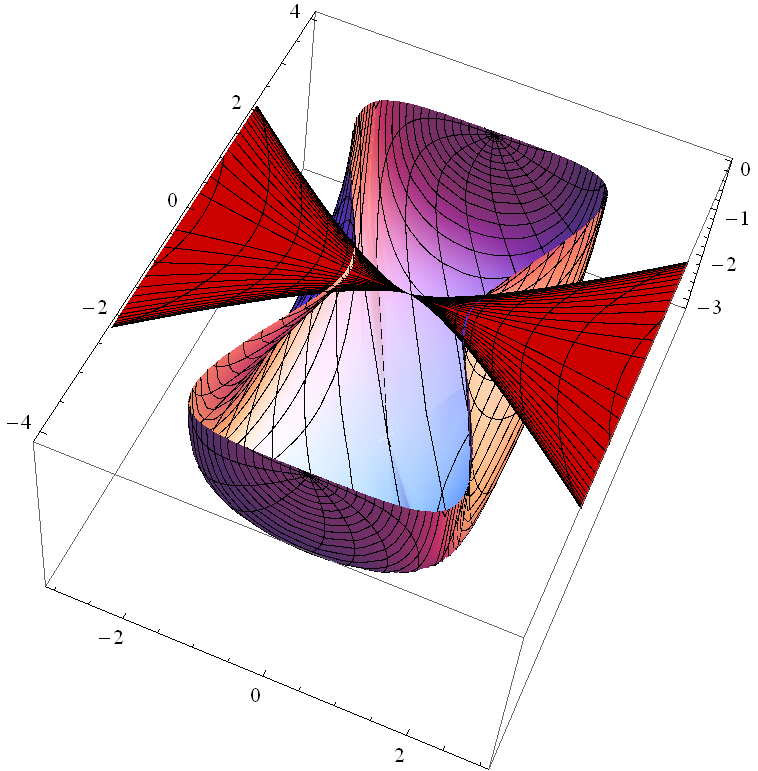}
\par\end{center}

\protect\caption{\label{fig:cutnegsphereRpiR12}Intersection of the cut locus with
the hemisphere $z<0$ of radius $R=\pi$ }
\end{minipage}\hfill{}%
\begin{minipage}[t]{0.45\columnwidth}%
~

\begin{center}
\includegraphics[scale=0.3]{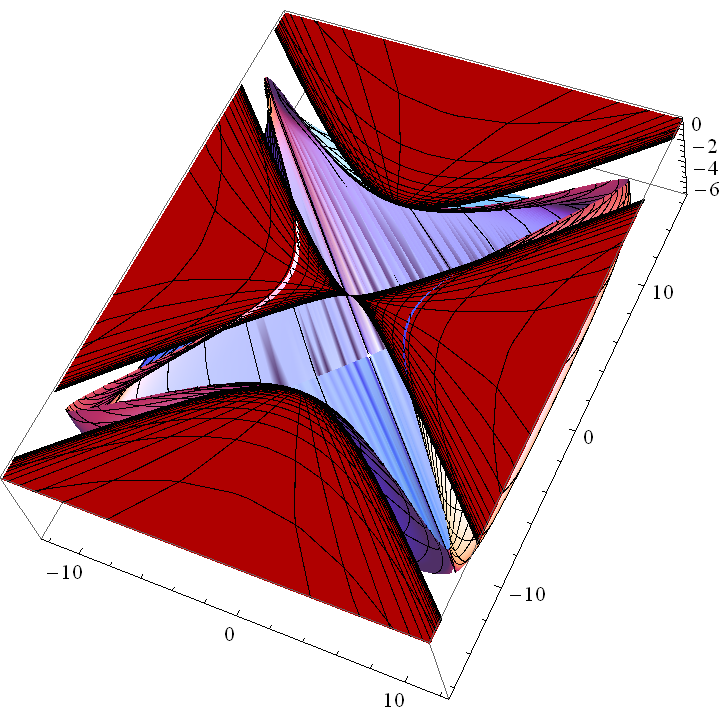}
\par\end{center}

\protect\caption{\label{fig:cutnegsphereR2piR12}Intersection of the cut locus with
the hemisphere $z<0$ of radius $R=2\pi$ }
\end{minipage}
\end{figure}
\begin{figure}
\centering{}%
\begin{minipage}[t]{0.45\columnwidth}%
\begin{center}
\includegraphics[scale=0.3]{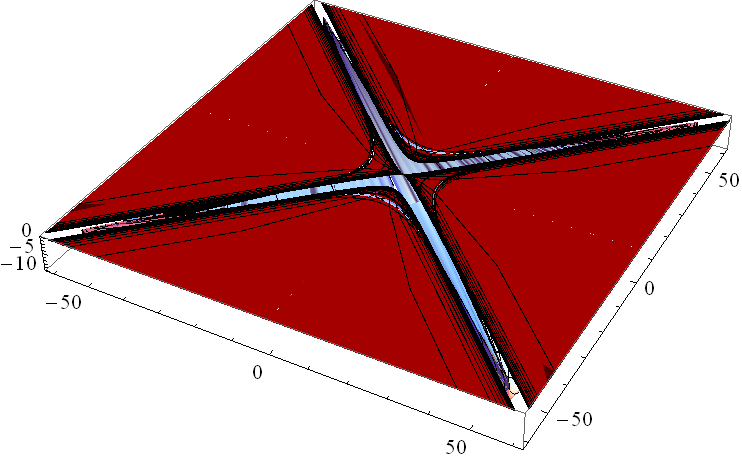}
\par\end{center}%
\end{minipage}\protect\caption{\label{fig:cutnegsphereR3piR12}Intersection of the cut locus with
the hemisphere $z<0$ of radius $R=2\pi$ }
\end{figure}
\begin{figure}
\centering{}%
\begin{minipage}[t]{0.45\columnwidth}%
\begin{center}
\includegraphics[scale=0.3]{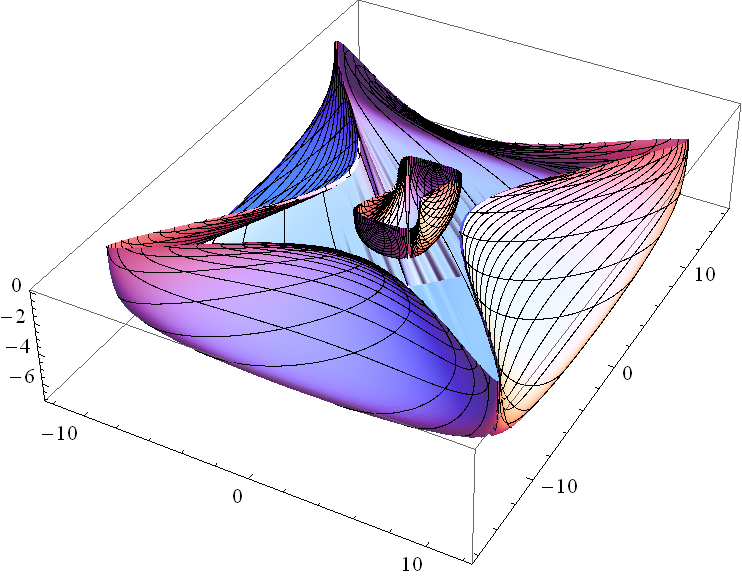}
\par\end{center}%
\end{minipage}\protect\caption{\label{fig:matr2R12}Matryoshka of hemispheres $z<0$ of radii $R=\pi$
and $R=2\pi$ }
\end{figure}

\section{Conclusion}

The global optimality analysis and structure of exponential mapping
for the sub-Riemannian problem on the Lie group SH(2) was considered.
We cutout open dense domains by Maxwell strata in the preimage and
in the image of exponential mapping and prove that restriction of
the exponential mapping to these domains is a diffeomorphism. This
fact leads to the proof that the cut time in the sub-Riemannian problem
on the Lie group $\mathrm{\ensuremath{SH(2)}}$ is equal to the first
Maxwell time. We then describe the global structure of the exponential
mapping and obtain a stratification of the cut locus in the plane
$z=0$. Consequently, the problem of finding optimal trajectories
from any initial point $q_{0}\in M$ to another point $q_{1}\in M,\quad z\neq0$
is reduced to solving a set of algebraic equations. Summing up, a
complete optimal synthesis for the sub-Riemannian problem on the Lie
group $\mathrm{SH}(2)$ was constructed. 
\bibliographystyle{unsrt}
\bibliography{ref}

\begin{thebibliography}{10}

\bibitem{Extremal_Pseudo_Euclid}
Y.~A. Butt{,} Yu. L. Sachkov{,} A.~I. Bhatti.
\newblock Extremal trajectories and {M}axwell strata in sub-{R}iemannian
  problem on group of motions of pseudo-{E}uclidean plane.
\newblock {\em Journal of Dynamical and Control Systems}, 20(3):341--364, July
  2014.

\bibitem{intg_SH2}
Y.~A. Butt{,} A. I. Bhatti{,} Yu.~L. Sachkov.
\newblock Integrability by quadratures in optimal control of a unicycle on
  hyperbolic plane.
\newblock In {\em American Control Conference}, Chicago Illionis, 1--3, Jul
  2015.

\bibitem{agrachev_sachkov}
A.~A. Agrachev{,} Yu.~L. Sachkov.
\newblock {\em Control Theory from the Geometric Viewpoint}.
\newblock Springer, 2004.

\bibitem{max_sre}
I.~Moiseev{,} Yuri~L. Sachkov.
\newblock Maxwell strata in sub-{R}iemannian problem on the group of motions of
  a plane.
\newblock {\em ESAIM: COCV}, 16:380--399, 2010.

\bibitem{cut_sre1}
Yuri~L. Sachkov.
\newblock Conjugate and cut time in the sub-{R}iemannian problem on the group
  of motions of a plane.
\newblock {\em ESAIM: COCV}, 16:1018--1039, 2010.

\bibitem{Sachkov_Dido_Comp_Max}
Yuri~L. Sachkov.
\newblock Complete description of the {M}axwell strata in the generalized
  {D}ido problem.
\newblock {\em English translation in Sbornik Mathematics}, pages 901--950,
  2006.

\bibitem{cut_engel}
A.~A. Ardentov{,} Yu.~L. Sachkov.
\newblock Cut time in sub-{R}iemannian problem on {E}ngel group.
\newblock {\em Accepted, ESAIM:COCV}, 2015.

\bibitem{Max_Conj_SH2}
Y.~A. Butt{,} Yu. L. Sachkov{,} A.~I. Bhatti.
\newblock Maxwell strata and conjugate points in the sub-{R}iemannian problem
  on the {L}ie group {SH}(2).
\newblock {\em arXiv:1408.2043v1}, 2014.

\bibitem{cut_sre2}
Yuri~L. Sachkov.
\newblock Cut locus and optimal synthesis in the sub-{R}iemannian problem on
  the group of motions of a plane.
\newblock {\em ESAIM: COCV}, 17:293--321, 2011.

\bibitem{Ja.Vilenkin}
N.~Ja. Vilenkin.
\newblock {\em Special Functions and Theory of Group Representations
  (Translations of Mathematical Monographs)}.
\newblock American Mathematical Society, revised edition, 1968.

\bibitem{sachkov_lectures}
Yuri~L. Sachkov.
\newblock Control theory on {L}ie groups.
\newblock {\em Journal of Mathematical Sciences}, 156(3):381--439, 2009.

\bibitem{Chow}
W.~L. Chow.
\newblock Uber {S}ysteme von linearen partiellen {D}ierentialgleichungen erster
  {O}rdnung.
\newblock {\em Mathematische Annalen}, 117:98--105, 1940.

\bibitem{Ravchevsky}
P.~K. Rashevsky.
\newblock About connecting two points of complete nonholonomic space by
  admissible curve.
\newblock {\em Uch Zapiski Ped}, pages 83--94, 1938.

\bibitem{diffeo}
Krantz{,} Parks.
\newblock {\em The implicit function theorem: history, theory and
  applications}.
\newblock Birkauser, 2001.

\end{thebibliography}

\end{document}